\documentclass[runningheads,11pt]{llncs}
\usepackage{amssymb}
\usepackage{amsmath}
\usepackage{graphicx}
\usepackage{subfigmat}
\usepackage{url}
\usepackage{amsfonts}
\usepackage{latexsym} 
\usepackage{wasysym}
\usepackage{newlfont}
\usepackage{leqno}

\usepackage{stmaryrd}

\usepackage{xspace}
\newcommand{\gismo}{{\fontfamily{phv}\fontshape{sc}\selectfont G\pmb{+}Smo}\xspace}

\newtheorem{assume}{Assumption}

\numberwithin{equation}{section}
    \urldef{\mailsa}\path|ulrich.langer@assoc.oeaw.ac.at|
    \urldef{\mailsb}\path|angelos.mantzaflaris@oeaw.ac.at|
    \urldef{\mailsc}\path|stephen.moore@ricam.oeaw.ac.at|
    \urldef{\mailsd}\path|ioannis.toulopoulos@oeaw.ac.at|
    \newcommand{\keywords}[1]{\par\addvspace\baselineskip  
     \noindent\keywordname\enspace\ignorespaces#1}
\usepackage{tcolorbox}
\definecolor{mycolor}{rgb}{0.122, 0.435, 0.698}
\makeatletter
\newcommand{\mybox}[1]{%
  \setbox0=\hbox{#1}%
  \setlength{\@tempdima}{\dimexpr\wd0+13pt}%
  \begin{tcolorbox}[colframe=mycolor,boxrule=0.5pt,arc=4pt,
      left=6pt,right=6pt,top=6pt,bottom=6pt,boxsep=0pt,width=\@tempdima]
    #1
  \end{tcolorbox}
}

\begin{document}
\mainmatter  

\title{Mesh Grading in Isogeometric Analysis}
\titlerunning{Mesh Grading in  dG IgA for Elliptic problems}

\author{
        Ulrich Langer
  \and  Angelos Mantzaflaris
  \and  Stephen E. Moore
  \and  Ioannis Toulopoulos
}


\author{
U. Langer 
  \and  A. Mantzaflaris 
  \and  St. E. Moore 
\and   I. Toulopoulos 
}
\authorrunning{ U. Langer, A. Mantzaflaris, S. E. Moore,  I. Toulopoulos }

\institute{ 
Johann Radon Institute for Computational and Applied Mathematics (RICAM)\\ 
		  of the Austrian Academy of Sciences\\
		  Altenbergerstr. 69, A-4040 Linz, Austria \\
\mailsa \\
\mailsb \\
\mailsc \\
\mailsd   
 }
\maketitle

\begin{abstract}

This paper  
is concerned with the construction of graded meshes 
for approximating so-called singular solutions 
of elliptic boundary value problems by means of 
multipatch
discontinuous Galerkin Isogeometric Analysis schemes.
Such solutions appear, for instance, 
in domains with re-entrant corners on the  boundary 
of the computational domain, 
in problems with changing boundary conditions, 
in interface problems, 
or in problems with singular source terms.
Making use of the 
analytic behavior of the solution, 
we construct the graded meshes in the neighborhoods of
such singular points following a 
multipatch
approach. 
We prove that 
appropriately graded meshes lead to the same convergence  
rates 
as
in the case of smooth solutions with approximately 
the same number of degrees of freedom.
Representative numerical examples are studied in order  
to confirm the  theoretical  convergence rates
and to demonstrate the efficiency of the 
mesh grading technology in Isogeometric Analysis.


\end{abstract}

\section{Introduction}
The gradient of the solution of elliptic boundary value problems 
can exhibit
singularities in the  vicinities of re-entrant corners or edges.
The same is true in case of changing boundary conditions or interface problems.
This singular behavior of the gradients was discovered and analyzed 
in the famous work by Kondrat'ev \cite{LMMT:Kondratev:1967a}.
We refer the reader to the monographs \cite{LMMT:Grisvard:1985a,LMMT:Grisvard:1992a,LMMT:KozlovMazyaRossmann2001}
for a more recent and comprehensive presentation of related results.
It is well known that these singularities may cause 
loss in
the approximation order of the standard discretization methods 
like the finite element method, 
see the classical monograph 
\cite{LMMT:StrangFix:1973a}
or the more recent paper
\cite{LMMT:ApelSandingWhiteman:1996}. 
In the case of  two dimensional problems with singular boundary  points, 
grading mesh  techniques have been 
developed for finite element methods in order to recover the 
full approximation order, see the classical textbook
\cite{LMMT:OganesjanRuchovetz:1979a}
and the more recent publications 
\cite{LMMT:ApelSandingWhiteman:1996,LMMT:ApelMilde1996,LMMT:FeistauerSaendig:2012a}, and 
\cite{LMMT:ApelHeinrich:1994a} for three-dimensional problems.
Here, we devise graded meshes for solving elliptic problems with singular solutions
by means of discontinuous Galerkin Isogeometric Analysis method (dG IgA).

\par
In the IgA frame, the use of 
B-splines or NURBS basis functions allow complicated CAD 
geometries to be exactly represented, and the key point of Hughes et al. \cite{LMMT:HughesCottrellBazilevs:2005a}
was to make use of the same basis to approximate the solution of the problem 
under consideration. Since this pioneer paper, 
applications of IgA method have been considered in many fields, see \cite{LMMT:Hughes_IGABook}.
Here, we apply a 
multipatch symmetric dG IgA method
which has been 
extensively studied 
for diffusion problems in volumetric computational domains 
and on surfaces
in \cite{LMMT:LangerToulopoulos:2014a} 
and \cite{LMMT:LangerMoore:2014a}, respectively,
see also \cite{LMMT:LangerMantzaflarisMooreToulopoulos:2014a}
for comprehensive presentation.
The  solution of the problem is 
independently approximated 
in every subdomain by IgA, 
without imposing any matching grid  conditions and without any  
continuity requirements for the discrete solution across the  subdomain interfaces.
Symmetrized          numerical fluxes with interior penalty  jump terms, 
see, e.g., \cite{LMMT:Dryja:2003a,LMMT:Riviere:2008a,LMMT:DiPietroErn:2012a}, 
        are introduced on the interfaces
         in order to treat the discontinuities of the discrete solution and to interchange information
         between the non matching grids.   As we will see later,  
     the consideration of the numerical scheme  in this general context
makes it more flexible to be applied on zone-type subdivisions of $\Omega$,  
which have been found to be quite convenient for treating elliptic 
boundary value problems
in domains with singular boundary points.   
%
\par 
This paper aims at the construction  of graded dG IgA meshes in the zones  located near
the  singular points      
in order to recover full 
convergence rates like in the case of smooth solutions on uniform meshes. 
The grading of the mesh 
 is mainly determined by 
the analytic behavior
 of the solution $u$ around the singular points 
 and follows the spirit of grading mesh techniques using layers, which have been proposed 
 for finite element methods  in 
\cite{LMMT:OganesjanRuchovetz:1979a,LMMT:ApelSandingWhiteman:1996,LMMT:ApelMilde1996}.
 According to this, having an  a priori knowledge about the location 
 of the singular point, e.g. the re-entrant corner, 
the domain $\Omega$ is
 subdivided into zones,
called layers in  \cite{LMMT:ApelSandingWhiteman:1996,LMMT:ApelMilde1996}, 
and  then a further subdivision of $\Omega$ into subdomains (also called patches in IgA), 
say $\cal{T}_{H}(\Omega):=\{\Omega_i\}_{i=1}^N$, is performed in such way 
that  $\cal{T}_{H}(\Omega)$ 
is 
in correspondence with 
the initial zone partition.    
On the other hand, 
 the solution can be split into a sum of a
 regular part $u_r\in W^{l\geq 2,2}(\Omega)$ 
and a singular part $u_s \in W^{1+\varepsilon ,2}(\Omega)$,  
with known $\varepsilon \in (0,1)$, 
i.e.,
$u=u_r+u_s$, see, e.g., \cite{LMMT:Grisvard:1985a}.
The analytical form of
 $u_s$ contains terms  with
singular exponents in the radial direction. 
 We use this  information and    construct
 appropriately graded  meshes in  the zones
  around the singular points. 
  The resulting graded meshes have  a ``zone-wise character'', 
  this means that the grid size of the graded mesh in every zone
 determines  the mesh of every subdomain which belongs into this zone, where 
  we assume that every subdomain
 belongs to only  one zone (the ideal situation is every zone to be a subdomain).
 We mention that the  mesh grading methodology is developed and is analyzed for the  classical  
 two dimensional problem  with a re-entrant  corner. 
 The proposed  methodology can be generalized and applied to other situations.
 This is shown by the numerical examples presented in Section~4. 
 
\par
The particular properties of the produced graded  meshes help
  us to  show optimal error estimates for the dG IgA method, 
 which exhibit optimal convergence rates.
 The error estimates for the proposed  method 
 are proved by using  a variation of C\'ea's Lemma and 
 using   B-spline quasi-interpolation estimates for
 $u\in W^{1,2}(\Omega)\cap W^{l\geq 2,p\in(1,2]}(\cal{T}_{H}(\Omega))$,  
 which have been proved  
 in \cite{LMMT:LangerToulopoulos:2014a}. More precisely, these interpolation estimates  have 
 subdomain character and  are expressed  with respect to the mesh size $h_i$ of the corresponding subdomain
 $\Omega_i$. For the domains away from the singular point,
 the solution is smooth (see $u_r$ part in previous splitting), 
and we can derive the usual interpolation estimates. 
Conversely, for the  subdomains $\Omega_i$, 
for which
the boundary  $\partial \Omega_i$ touches the singular point,
the singular part $u_s$ of the solution $u$ 
can be considered  as a function from the Sobolev space $W^{2,p\in(1,2)}(\Omega_i)$.
%
%
Now the estimates given in \cite{LMMT:LangerToulopoulos:2014a} enable us to derive 
error estimates for the singular part $u_s$. 
This makes the whole error analysis easier 
in comparison with the techniques earlier developed 
for the finite element method, e.g., in
\cite{LMMT:ApelSandingWhiteman:1996,LMMT:Apel_1999_Int_NonSmooth,LMMT:FeistauerSaendig:2012a}.

\par
We mention that, in the literature, other IgA techniques have been proposed 
for solving two-dimensional problems with  singularities very recently.  
In \cite{LMMT:OhKimJeong:2013a} and \cite{LMMT:JeongOhKangKim:2013a},
the mapping technique has been developed, where 
the original  B-spline finite dimensional space has been  enriched by  generating  singular  
functions which resemble the types of the singularities of the problem. 
The mappings
constructed on this enriched space describe the geometry singularities explicitly. 
Also in \cite{LMMT:SangalliaNurbs2012}, by studying the anisotropic character of the 
singularities of the problem, the one-dimensional approximation properties of
the B-splines are generalized for two-dimensional problems, 
in order to produce anisotropic refined meshes in the regions of the singular
points.

\par
The rest of the  paper is organized as follows. 
The problem description, the weak formulation and 
the dG IgA discrete analogue are presented in Section 2. 
Section 3 discusses  the construction of the 
appropriately graded IgA meshes, 
and provides the proof for obtaining  the full approximation order of 
the dG IgA method on the graded meshes.
Several two and three dimensional examples are presented in Section 4. 
Finally, we draw some conclusion.

\section{Problem description and dG IgA discretization}
\label{sec:ProblemDescriptionAndDiscretization}
First, let us introduce some notation. 
We define the \textit{differential operator}
\begin{align}\label{00.0}
 D^a=D_1^{\alpha_1}\cdot\cdot\cdot D_d^{\alpha_d}, \text{with}
 {\ }D_j=\frac{\partial}{\partial x_j}, D^{(0,...,0)}u=u,
\end{align}
where $\alpha=(\alpha_1,...,\alpha_d)$, with $\alpha_j\geq 0, j=1,...,d$,
denotes a multi-index of the degree $|\alpha| = \sum_{j=1}^d\alpha_j$.
For a bounded Lipschitz domain $\Omega \subset \mathbb{R}^{d}$, $d=2,3$
we  denote by $W^{l,p}(\Omega)$, with $l \geq 1$  and $1\leq p \leq \infty$,
the usual Sobolev function spaces   endowed with the norms
\begin{subequations}
\begin{align}\label{00.1}
 \|u\|_{W^{l,p}(\Omega)} = \big(\sum_{0\leq |\alpha| \leq m} \|D^{\alpha}u\|_{L^p(\Omega)}^p\big)^{\frac{1}{p}},\\
 \|u\|_{W^{l,\infty}(\Omega)} = max_{0\leq |\alpha| \leq m} \|D^{\alpha}u\|_{\infty}.
\end{align}
\end{subequations}
More details 
about Sobolev's function spaces
can be found in  \cite{LMMT:Adams_Sobolevbook}.
We often write $a\sim b$, meaning that $C_m a \leq b \leq C_M a$, with $C_m $ and $ C_M$ are positive
constants independent of the discretization parameters. 

\subsection{The model problem}
Let us assume that the boundary of $\Gamma_D=\partial \Omega$ of $\Omega$
contains geometric singular parts. 
In particular, for $d=2$, we consider domains  which have corner
boundary points with internal angles greater than $\pi$.  
For $d=3$, we consider that case where                      
the domain $\Omega$ can be  
described
in the form
$\Omega = \Omega_2 \times Z$,
where $\Omega_2 \subset \mathbb{R}^2$ and $Z=[0,z_M]$ is an interval.
The cross section of $\Omega$ has only one corner with an interior angle $\omega \in (\pi,2\pi)$.
This means that the $\partial \Omega$ has only one singular edge 
which is $\Gamma_s:=\{(0,0,z), 0\leq z \leq z_M\}$.
The remaining parts  of $\Gamma_D$ are considered as smooth,
see Fig. \ref{fg1_3dEdge}(a) and Fig. \ref{fg1_3dEdge}(b)    
for an illustration of the domains.
\par
For simplicity, we restrict our study to the following model problem
\begin{equation} \label{0}
  -\mathrm{div}(\alpha \nabla u) = f \quad \text{in} \quad \Omega, \quad u  = u_D  
    \quad \text{on} \quad \partial \Omega,
\end{equation}
%
where the coefficient  $ \alpha (x) \in  L^{\infty}(\Omega)$ 
is a piecewise constant function, bounded from above
and below by  positive constants, 
 $f\in L^2(\Omega)$ and $ u_D\in H^{\frac{1}{2}}(\partial \Omega)$ are given data.  
 The variational formulation of \eqref{0} reads as follows: 
find $u\in W^{1,2}(\Omega)$ 
 such that $u=u_D$ on $\Gamma_D=\partial \Omega$ and 
\begin{subequations}\label{0.0}
 \begin{align}\label{0.0a}
  a(u,v)=l(v), \quad \forall v\in W^{1,2}_0(\Omega),\qquad\\
  \intertext{where}
  \label{0.0b}
  a(u,v)=\int_{\Omega}\alpha \nabla u\cdot \nabla v\,dx
  \quad \mbox{and} \quad l(v)=\int_{\Omega}fv\,dx.
 \end{align}
\end{subequations}
It is clear that, under the assumptions made above, there exists
a unique solution of the variational problem \eqref{0.0}
due to Lax-Milgram's lemma.


\par 
We follow the theoretical analysis of the regularity 
of solution presented in \cite{LMMT:Grisvard:1992a}.  
We consider the two-dimensional case. Suppose that the  $\Gamma_D$ has only one
singular corner, say $P_s$,   with  internal angle $\omega \in (\pi,2\pi)$, 
and that the boundary parts from the one and the other side of $P_s$ are straight lines,
see Fig.~\ref{fg1_3dEdge}(a).
We consider the local cylindrical coordinates $(r,\theta)$ with origin $P_s$,
and define the cone (a circular sector with angular point $P_s$).
\begin{align}\label{0.1}
   \hspace*{-5mm}{
   \cal{C}=\{(x,y)\in \Omega: x=r\cos(\theta),y=r\sin(\theta), 0<r<R,0<\theta<\omega\}.
   }
\end{align}
We construct a highly smooth cut-off function $\xi$ in $\cal{C}$, 
such that $\xi\in C^{\infty}$, 
and it is supported inside the cone $\cal{C}$. 
It has been shown in \cite{LMMT:Grisvard:1992a}, that the solution $u$ of the problem
(\ref{0.0})  
can be  written as a sum of a regular function $u_r\in W^{l\geq 2,2}(\Omega)$ 
and a singular function $u_s$,
\begin{align}\label{0.4}
 u=u_r+u_{s},
\end{align}
with
 \begin{align}\label{0.3}
   u_s=\xi(r)\gamma r^{\lambda}\sin(\lambda \theta),
  \end{align}
where $\gamma$ is the \textit{stress intensity factor} 
(for the two-dimensional problems   is a real number depending only on $f$)
and  $\lambda=\frac{\pi}{\omega}\in (0,1)$ is
an exponent which determines the strength of the singularity. 
Since $\lambda <1$, by an easy computation, we can show that
the singular function $u_s$ does not belong to $W^{2,2}(\Omega)$ 
but to $W^{l=2,p}(\Omega)$ with
$p={2}/{(2-\lambda)}$. 
Consequently, the
regularity properties of $u$ in $\cal{C}$ are mainly determined by the regularity properties of $u_s$,
and we can 
assume
that 
$u\in W^{1,2}(\Omega)\cap W^{l,p}(\cal{T}_H(\Omega))$, (see below details for the $\cal{T}_H(\Omega)$).  
\begin{remark}
For the expression (\ref{0.3}), we admit that the computational domain has only 
one non-convex corner and 
only Dirichlet boundary conditions are prescribed on $\partial \Omega$. 
Similar expression can be derived if there are   more  non-convex corners 
and if there are other type of boundary conditions,
see  details in \cite{LMMT:Grisvard:1992a}.
\end{remark}

\subsection{The dG IgA discrete scheme}
\subsubsection{Isogeometric Analysis Spaces}
We assume a non-overlapping subdivision 
$\mathcal{T}_{H}(\Omega) :=\{\Omega_i\}_{i=1}^N$ of the computational domain $\Omega$ such that
$\bar{\Omega}=\bigcup_{i=1}^N\bar{\Omega}_{i}$ with $\Omega_i\cap \Omega_j=\emptyset$ for $i\neq j$.
The subdivision  $\mathcal{T}_{H}(\Omega)$ is considered to be compatible 
with the discontinuities of the coefficient $\alpha$, 
i.e.,
the jumps can only appear  on the interfaces 
$F_{ij}=\partial \Omega_i \cap \partial\Omega_j$ between the subdomains.
For the sake of brevity in our notations,  the set of common interior faces 
are denoted by $\mathcal{F}_{I}$. 
The collection of the  faces  that belong to  
$\partial \Omega$ are denoted  by $\mathcal{F}_B$, 
i.e., $F\in \mathcal{F}_B$,
if there is a $\Omega_i\in \mathcal{T}_{H}(\Omega)$ such that $F=\partial \Omega_i \cap \partial \Omega$.
We denote the set of all subdomain faces by $\mathcal{F} = \mathcal{F}_{I} \cup \mathcal{F}_B.$ 

\par
In the multi-patch (multi-subdomain) IgA  context, 
each subdomain is represented by a B-spline (or NURBS) mapping. To accomplish this, 
we associate  each $\Omega_i$ with a vector of knots 
$\mathbf{\Xi}^d_i=(\Xi_i^1,...,\Xi_i^{\iota},...,\Xi_i^d)$,
with
$\Xi_i^\iota=\{\mathbf{\xi}^\iota_1,\mathbf{\xi}_2^\iota,...,\xi_{n}^\iota\}$, 
$\iota = 1,\ldots,d$, 
which are set on the parametric domain $\widehat{\Omega}=(0,1)^d$. 
The interior knots of $\mathbf{\Xi}^d_i$ are considered without  repetitions and form 
a mesh  $T^{(i)}_{h_i,\widehat{\Omega}}=\{\hat{E}_m\}_{m=1}^{M_i}$ in $\widehat{\Omega}$,
where $\hat{E}_m$ are the micro elements. 
Given a micro element $\hat{E}_m\in T^{(i)}_{h_i,\widehat{\Omega}} $, we denote by 
$h_{\hat{E}_m}=diameter(\hat{E}_m)$, and the 
local grid size $h_i$ is defined to be the maximum diameter
of all $\hat{E}_m\in T^{(i)}_{h_i,\widehat{\Omega}}$, that is $h_i= \max\{h_{\hat{E}_m}\}$.
We refer the reader to \cite{LMMT:Hughes_IGABook} for more information 
about the meaning of the knot vectors in CAD and IgA.
%
\begin{assume}\label{assume1}
	The  meshes $T^{(i)}_{h_i,\widehat{\Omega}}$ defined by the knots $\mathbf{\Xi}^d_i$ 
	are quasi-uniform, i.e.,
	there exist a constant $\sigma \geq 1$ such that 
	$\sigma^{-1} \leq \frac{h_{\hat{E}_m}}{h_{\hat{E}_{m+1}}} \leq \sigma$.\\
\end{assume}
On each $T^{(i)}_{h_i,\widehat{\Omega}}$,  we derive the finite dimensional 
space $\hat{\mathbb{B}}^{(i)}_{h_i}$ spanned by  
 B-spline (or NURBS) basis functions of degree $k$, see more details in 
 \cite{LMMT:Hughes_IGABook,LMMT:BazilevsBeiraoCottrellHughesSangalli:2006a,LMMT:Schumaker_Bspline_book},
\begin{subequations}\label{0.00b}
\begin{align}\label{0.00b1}
 \hat{\mathbb{B}}^{(i)}_{h_i}=
span\{\hat{B}_\mathbf{j}^{(i)}(\hat{x})\}_{\mathbf{j}=0}^{dim(\hat{\mathbb{B}}_{h_{i}}^{(i)})}. 
\intertext{Every $\hat B_\mathbf{j}^{(i)}(\hat{x})$ function  in (\ref{0.00b1}) 
      is derived by means of tensor
           products of  one-dimensional B-spline basis functions, i.e.}
 \label{0.00b2}
\hat B_{\mathbf{j}}^{(i)}(\hat{x}) = \hat B_{\iota=1,j_1}^{(i)}(\hat{x}_1)\cdot\cdot\cdot \hat B_{\iota=d,j_d}^{(i)}(\hat{x}_d).
\end{align}
\end{subequations}
 In the following, we suppose that  the one-dimensional B-splines in (\ref{0.00b2}) have the same degree $k$. 
 Finally, having the B-spline spaces, we can represent each  subdomain $\Omega_i$
by the parametric mapping
\begin{subequations}\label{0.0c}
\begin{align}
\label{0.0c1}
 \mathbf{\Phi}_i: \widehat{\Omega} \rightarrow \Omega_i, &\quad
 \mathbf{\Phi}_i(\hat{x}) = \sum_{\mathbf{j}}C^{(i)}_\mathbf{j} \hat{B}_\mathbf{j}^{(i)}(\hat{x}):=x\in \Omega_i,\\
 \label{0.0c2}
 \text{with}&\quad \hat{x} = \mathbf{\Psi}_i(x):=\mathbf{\Phi}^{-1}_i(x),
 \end{align}
\end{subequations}
   where $C_\mathbf{j}^{(i)}$ are the B-spline control points, $i=1,...,N$,
cf. \cite{LMMT:Hughes_IGABook}.
 \par  
 We construct a mesh $T^{(i)}_{h_i,\Omega_i} =\{E_{m}\}_{m=1}^{M_i}$
 for every $\Omega_i$, whose vertices are the images of the vertices
 of the corresponding parametric mesh $T^{(i)}_{h_i,\widehat{\Omega}}$
 through $\mathbf{\Phi}_i$.  Notice that, the above 
 subdomain mesh construction  can result in non-matching
 meshes along the patch interfaces.
\par
 Further, by taking advantage of  the properties of $\mathbf{\Phi}_i$,
 we  define 
the global finite dimensional B-spline (dG) space
 \begin{subequations}\label{0.0d}
 \begin{align}\label{0.0d1}
\mathbb{B}_h(\mathcal{T}_{H}):=\mathbb{{B}}^{(i)}_{h_i}(\Omega_i)\times ...\times \mathbb{{B}}^{(N)}_{h_N}(\Omega_N), 
\intertext{where every  $\mathbb{B}^{(i)}_{h_{i}}(\Omega_i)$ is defined 
on $T^{(i)}_{h_i,\Omega_i}$ as follows:}
\label{0.0d2}
 \mathbb{B}^{(i)}_{h_{i}}(\Omega_i):=\{B_{\mathbf{j}}^{(i)}|_{\Omega_i}: B_\mathbf{j}^{(i)}({x})=
  \hat{B}_\mathbf{j}^{(i)}\circ \mathbf{\Psi}_i({x}),{\ }\text{for}{\ }
  \hat{B}_\mathbf{j}^{(i)}\in \hat{\mathbb{B}}^{(i)}_{h_i} \}. 
\end{align}
\end{subequations}
Later, the solution $u$ of the problem (\ref{0.0}) will be approximated by the discrete (dG) solution
$u_h \in \mathbb{B}_h(\mathcal{T}_{H})$.
\begin{figure}
 \begin{subfigmatrix}{1}
\includegraphics[width=0.45\textwidth]{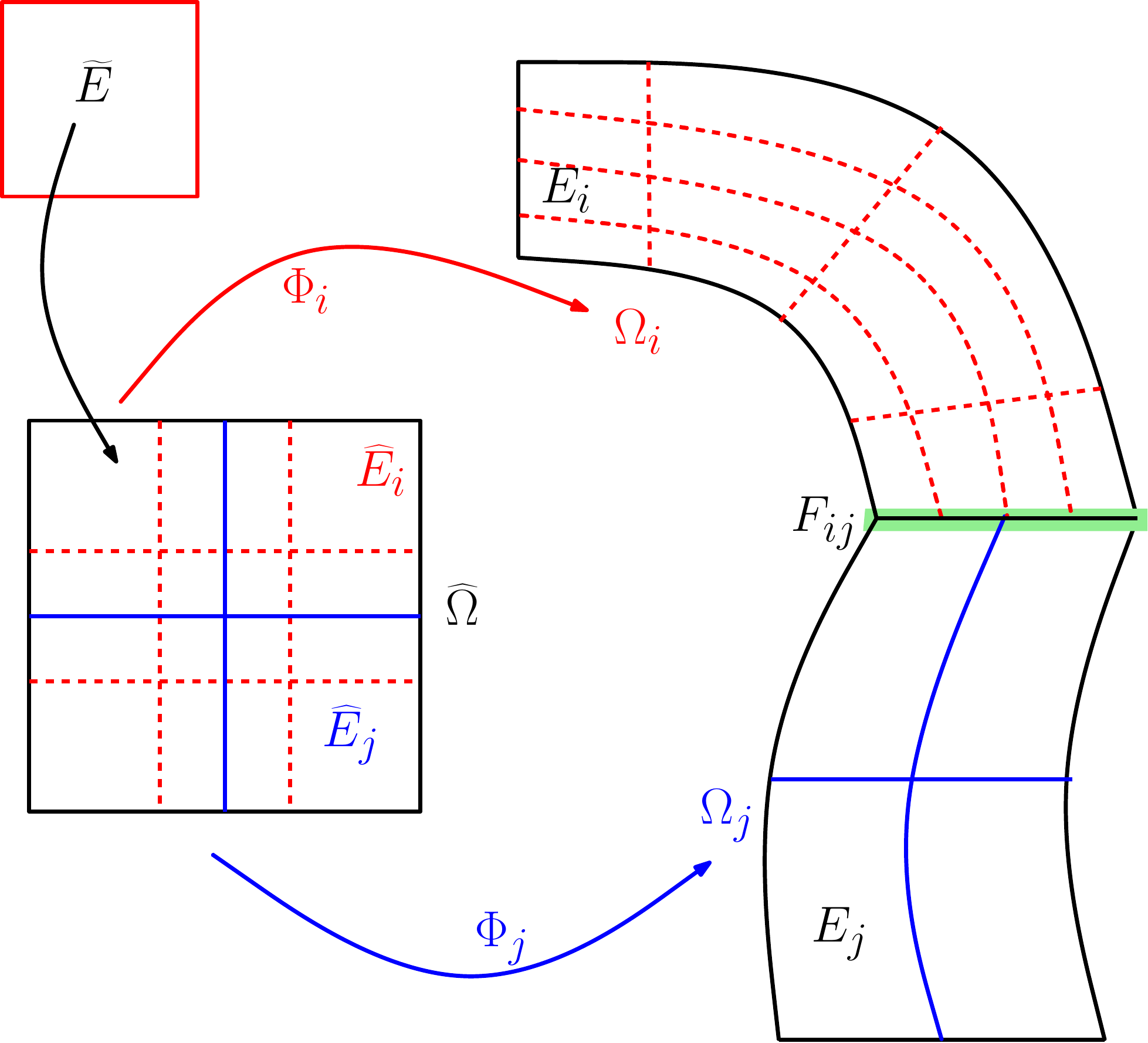}
 \end{subfigmatrix}
 \caption{The parametric domain  and two adjacent subdomains  with different underlying
          meshes red and blue.}
 \label{fg1_Domains}
\end{figure}
 \subsubsection{Discrete Problem}
 The problem (\ref{0.0}) is independently discretized in every $\Omega_i$ using the  spaces (\ref{0.0d2})
 without imposing   continuity requirements for the B-spline basis functions 
 on the interfaces $F_{ij}=\partial \Omega_i \cap \partial \Omega_j$ 
 and also  non-matching grids may exist. 
 Using the notation $\phi_h^{(i)}:=\phi_h|_{\Omega_i}$,
we define the average and the jump of $\phi_h \in \mathbb{B}_h(\mathcal{T}_{H})$ on $F_{ij}\in\mathcal{F}_I$  by
\begin{subequations}\label{5a}
\begin{align}
   \{\phi_h\}:=\frac{1}{2}(\phi_h^{(i)}+\phi_h^{(j)}), & {\ }\text{and}\qquad\llbracket \phi_h \rrbracket :=\phi_h^{(i)} - \phi_h^{(j)},\\
   \intertext{ and, for $F_i\in \mathcal{F}_B$,}
   \label{3.16c}
  \{\phi_h\}:=\phi_h^{(i)}, &{\ }\text{and}\qquad\llbracket \phi_h\rrbracket :=
  \phi_h^{(i)}.
\end{align}
\end{subequations}
 
The discrete problem is specified by the symmetric dG IgA method, 
see \cite{LMMT:LangerToulopoulos:2014a}, and 
reads
as follows: find 
$u_h\in \mathbb{B}_h(\mathcal{T}_{H})$
such that
\begin{subequations}\label{3}
\begin{flalign}\label{3a}
 a_h(u_h,\phi_h)=&l(\phi_h)+p_D(u_D,\phi_h),{\ } \forall \phi_h \in \mathbb{B}_h(\mathcal{T}_{H}),
 \intertext{where the dG bilinear form is given by}
 \label{3b}
 a_h(u_h,\phi_h) = & \sum_{i=1}^N \Big( a_i(u_h,\phi_h)-\sum_{F_{ij}\subset \partial\Omega_i}\big(\frac{1}{2}s_i(u_h,\phi_h)+p_i(u_h,\phi_h)\big)\Big)
\intertext{ with the bilinear forms (cf. also \cite{LMMT:Dryja:2003a}):}
 \label{3c}
 a_i(u_h,\phi_h) =& \int_{\Omega_i}\alpha\nabla u_h\nabla\phi_h\,dx, \\
 \label{3d}
 s_i( u_h,\phi_h)=& \int_{F_{ij}} \{\alpha\nabla u_h\}\cdot \mathbf{n}_{F_{ij}} \llbracket \phi_h \rrbracket +
                                                 \{\alpha\nabla \phi_h\}\cdot \mathbf{n}_{F_{ij}} \llbracket u_h \rrbracket\,ds, \\
 \label{3e}
 p_i(u_h,\phi_h) =&
                   \begin{cases}
 	                \int_{F_{ij}} \Big(\frac{\mu \alpha^{(j)}}{h_j}+\frac{\mu \alpha^{(i)}}{h_i}\Big) \llbracket  u_h \rrbracket \llbracket \phi_h \rrbracket\,ds,&\text{if}{\ } F_{ij}\in \mathcal{F}_I, \\
 	                \int_{F_{i}} \frac{\mu \alpha^{(i)}}{h_i} \llbracket  u_h \rrbracket \llbracket \phi_h \rrbracket\,ds,&\text{if}{\ } F_{ij}\in \mathcal{F}_B, \\
                   \end{cases}\\
 \label{3f}
 p_D(u_D,\phi_h)=&  \int_{F_{i}} \frac{\mu \alpha^{(i)}}{h_i}  u_D  
\phi_h\,ds, {\ }\quad F_{i}\in \mathcal{F}_B.
\end{flalign}
\end{subequations}
Here the unit normal vector  $\mathbf{n}_{F_{ij}}$ is oriented from $\Omega_i$ 
towards the interior of $\Omega_j$.
The penalty parameter $\mu>0$ must be chosen large enough 
in order to ensure the stability of the dG IgA method 
\cite{LMMT:LangerToulopoulos:2014a}.  
%
%
%
%
%
%
%
%
  \begin{figure}
    \begin{subfigmatrix}{3}
     \subfigure[]{\includegraphics[scale=0.24]{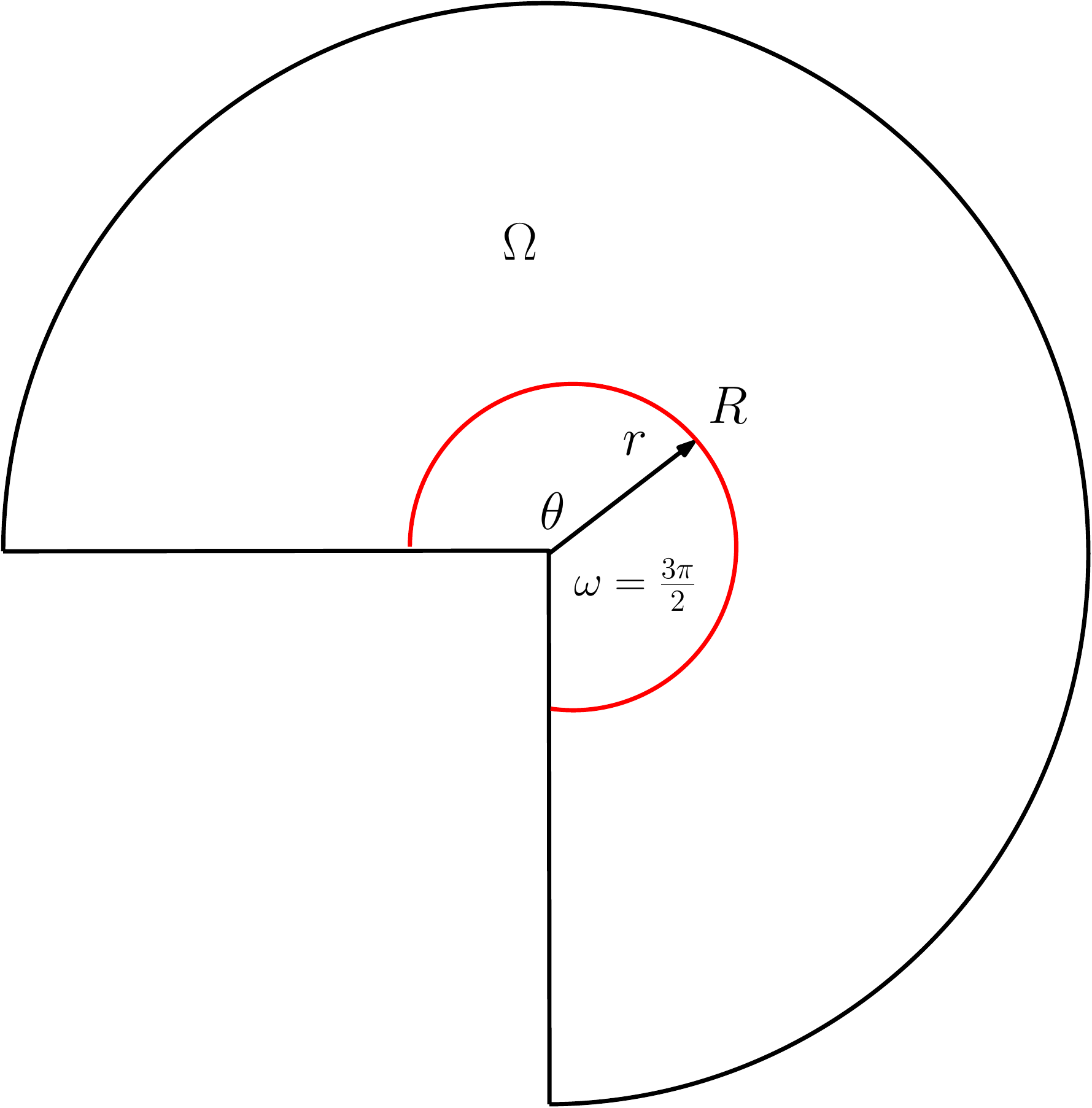}}
     \subfigure[]{\includegraphics[scale=0.45]{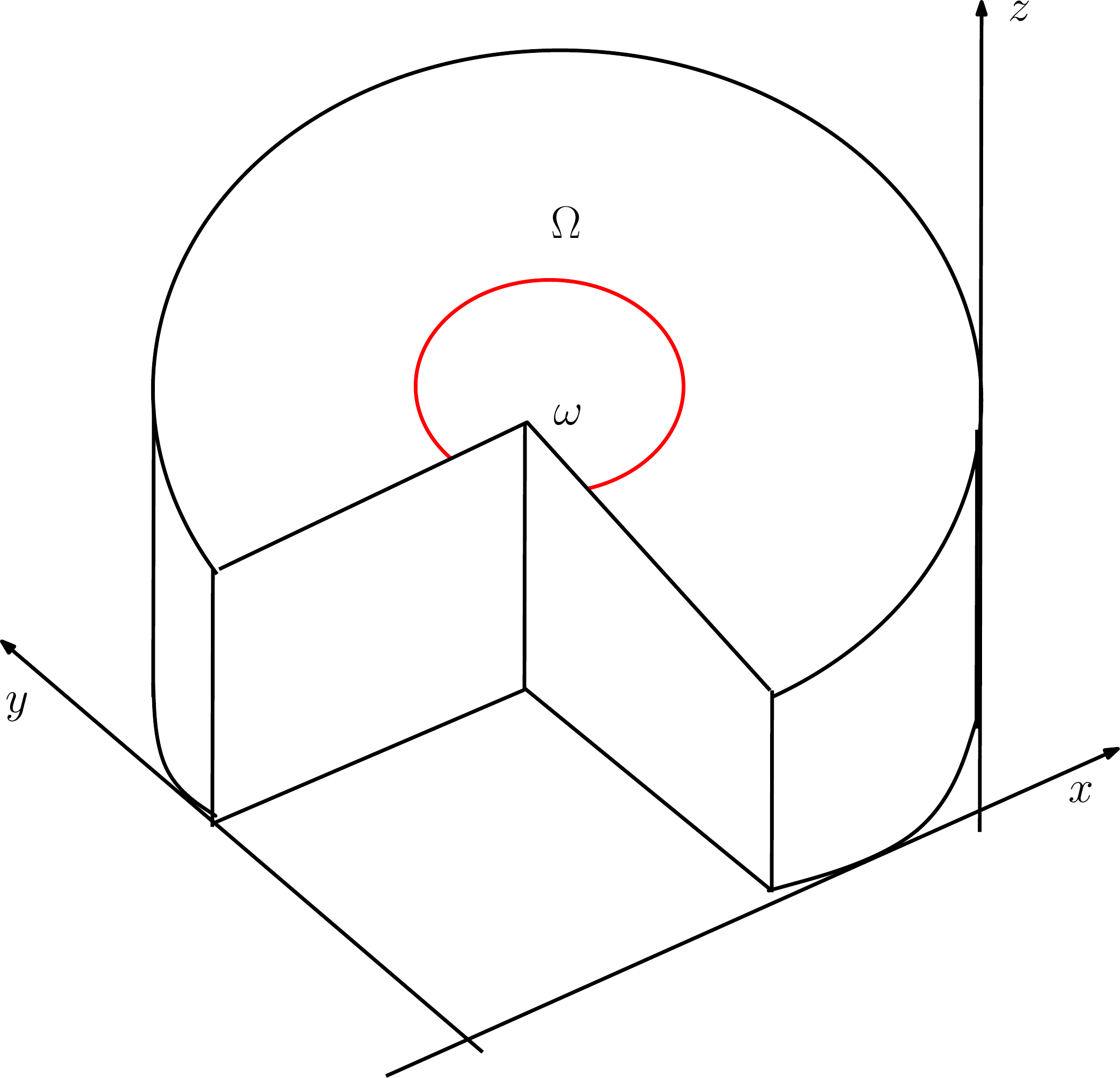}}
     \subfigure[]{\includegraphics[scale=0.24]{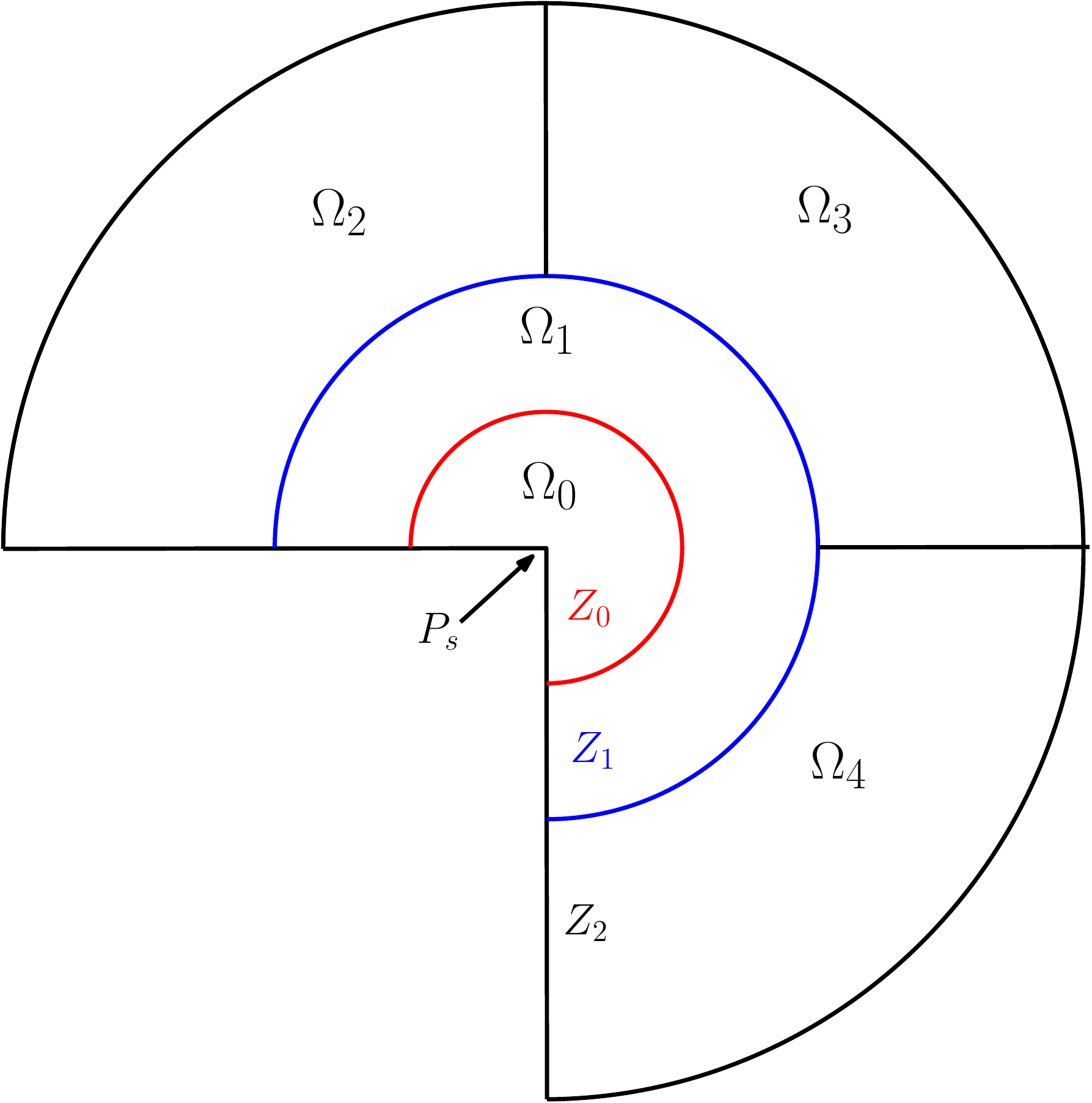}}
    \end{subfigmatrix}
 \caption{The domains, (a) two-dimensional with corner singularity, (b)
          three-dimensional with re-entrant edge,
          (c) subdivision of $\Omega$ into zones and subdomains.}
 \label{fg1_3dEdge}
\end{figure}

\section{IgA on graded meshes}

In many realistic applications,
we very often have to solve problems similar to (\ref{0.0}) in domains with non-smooth boundary parts, 
that possess geometric singularities, for instance, non-convex corners, see Fig. \ref{fg1_3dEdge}.
It is well-known that the numerical methods loose accuracy when they are applied to this type
of problems. 
This occurs as a result of the reduced regularity of  the solutions in the vicinity  of the
non-smooth parts \cite{LMMT:Grisvard:1985a}. 
  When finite element methods are used,    graded  meshes have been  utilized
   around the singular boundary parts  in order to obtain optimal convergence rates, 
    see e. g.
  \cite{LMMT:ApelSandingWhiteman:1996,LMMT:ApelHeinrich:1994a}, 
  see also \cite{LMMT:FeistauerSaendig:2012a} for dG methods. 
  The basic idea of this grading mesh technique is to use the 
  a priori knowledge of the singular behavior of the solution around the singular boundary points, 
cf. (\ref{0.3}) and (\ref{0.4})),
and consequently adjust accordingly the size of the elements. 
  
\par
The purpose of this paper is to extend  the grading mesh techniques from the finite element method
to dG IgA framework for solving  boundary value problems like  (\ref{0.0})  
in the presence of singular points. 
We develop a mesh grading algorithm  around the singular boundary parts 
inspired by the grading mesh methodology using layers, therefore, 
extending the approach used in finite element methods, 
cf. \cite{LMMT:ApelSandingWhiteman:1996,LMMT:ApelMilde1996}, to isogeometric analysis.
Next, we construct the graded mesh and show that the proposed dG IgA method exhibits 
optimal  convergence rates as for the problems with high regularity  solutions.
We present our mesh grading technique and the corresponding analysis for two-dimensional problems.
In Section~4, we also apply our methodology to some three-dimensional examples  
and discuss the numerical results.
\subsection{A priori mesh grading} \label{apriorimesh}
The grading of the  meshes around the singular points is guided by the  exponent $\lambda$, which
specifies the regularity of the function $u_s$, see (\ref{0.3}), 
and 
by  the location of the singular boundary point too.  
 Next, we discuss the construction of the mesh for the case of one singular geometric point on 
$\partial \Omega$.

\par
Let $P_s$ be the singular point and let 
$U_s:=\{x\in \Omega: |P_s-x| \leq R= L_{U} h, \;\mbox{with}\;   L_{U}\geq 2\}$ 
 be  an area around $P_s$ in $\Omega$, which
 is further  subdivided  into    
$\zeta_M$ ring-type zones $Z_{\zeta}$, $\zeta=0,..,\zeta_M$,
such that the distance from $P_s$ is $D_{(Z_{\zeta},P_s)}:=C(n_\zeta h)^{\frac{1}{\mu}}$,  
where $C=R^{1-\frac{1}{\mu}}$ and
$0\leq  n_\zeta < L_{U} $.  By $\mu \in (0,1]$, 
 we denote the grading control parameter.
 The  radius of every zone is defined to be 
 $R_{Z_{\zeta}}:=D_{(Z_{\zeta+1},P_s)} - D_{(Z_{\zeta},P_s)}=C(n_{\zeta+1} h)^{\frac{1}{\mu}} -
 C(n_{\zeta} h)^{\frac{1}{\mu}}$, where we suppose that there is a $\nu>0$
 such that   $n_{\zeta+1}=n_{\zeta}+\nu$ with 
   $1 \leq \nu < L_{U}-1$. In particular,  we set $R_{Z_{M}}=R-D_{(Z_{M-1},P_s)}$.
 \par
 For convenience, we assume that the initial subdivision $\mathcal{T}_{H}(\Omega)$ fits to the 
   $Z_{\zeta}$ ring  zone partition in order to 
 fulfill the following conditions, for an illustration, see 
   Fig. \ref{fg1_3dEdge}(c) with $\zeta_M=3$:
\begin{itemize}
	\item The subdomains can be grouped into those which belong (entirely) into the area $U_s$ and those
	that belong (entirely) into $\Omega\setminus U_s$. This means that there is no $\Omega_i,{\ } i=1,...,N$
         such that $U_s\cap \Omega_i\neq \emptyset$ and $(\Omega\setminus U_s)\cap \Omega_i\neq \emptyset$.
        \item Every ring zone $Z_{\zeta}$  is partitioned into ``circular'' subdomains $\Omega_{i_{\zeta}}$, which
              have radius $R_{\Omega_{i_{\zeta}}}$ equal to the radius of the zone,
              that is $R_{\Omega_{i_{\zeta}}}= R_{Z_{\zeta}}$. For
              computational efficiency reasons, we prefer, if it is possible, 
              every zone to be only represented by one subdomain. 
	      This essentially depends  on the characteristics of the problem, 
               i.e., the shape of $\Omega$ and the coefficient $\alpha$.  
        \item The zone $Z_{0}$ is represented by one subdomain, say $\Omega_{i_0}$, and the mesh 
               $T^{({i_0})}_{h_{i_0}}(\Omega_{i_0})$ includes all the micro-elements $E$  such that $\partial E \cap P_s \neq \emptyset$. 
\end{itemize}
 We construct the  meshes $T^{(i_{\zeta})}_{h_{i_{\zeta}}}(\Omega_{i_{\zeta}})$   
 (we will explain later how we can choose the grid size)
in order to satisfy the following properties: 
for $\Omega_{i_{\zeta}}$  with distance $D_{(Z_{\zeta},P_s)}$ from $P_s$,
 the mesh size  $h_{i_{\zeta}}$  
is defined to be $h_{i_{\zeta}}=\cal{O}(h R_{\Omega_{i_{\zeta}}}^{1-\mu})$ and 
for  $T^{({i_0})}_{h_{i_0}}(\Omega_{i_0})$ the mesh size is  
of order $h_{i_{0}}=\cal{O}(h^{\frac{1}{\mu}}).$ 
Thus, we have the following relations:
\begin{subequations}\label{3.2}
\begin{align}\label{3.2a}
 C_m h^{\frac{1}{\mu}} &\leq h_{{i_{\zeta}}}   \leq C_M h^{\frac{1}{\mu}}, &\text{if}{\ }  \overline{\Omega}_{i_{\zeta}}\cap P_s\neq \emptyset, \\
\label{3.2b}
 C_m hR_{\Omega_{i_{\zeta}}}^{1-\mu} &        \leq h_{i_\zeta}   \leq C_M h D_{(Z_{\zeta},P_s)}^{1-\mu}, &\text{if}{\ }  \overline{\Omega}_{i_\zeta}\cap P_s= \emptyset. 
\end{align}
\end{subequations}
We need to specify the mesh size for 
every $T^{({i_\zeta})}_{h_{i_\zeta}}(\Omega_{i_\zeta})$ in order to satisfy   
inequalities (\ref{3.2}).
We set  the mesh size of  $T^{({i_\zeta})}_{h_{i_\zeta}}(\Omega_{i_\zeta})$ to be
of order 
$h_{i_\zeta}=\cal{O}(R_{Z_{\zeta}}\nu^{-(1/\mu}))$, 
and, for a uniform subdomain mesh,
we can set 
%
\begin{equation*}
h_{i_\zeta}=C\frac{(n_{\zeta}+\nu)h)^{\frac{1}{\mu}}-(n_{\zeta} h)^{\frac{1}{\mu}}}{int(\nu^\frac{1}{\mu})} 
\end{equation*}
where $C=R^{1-\frac{1}{\mu}}$ and 
$int(\nu^{-\mu})$
denotes the nearest  integer to
$\nu^{-\mu}$.
Notice that
the grading has ``a subdomain   character'' 
and is  mainly determined by the  parameter $\mu \in (0,1]$. 
For $\mu=1$, we get $h_{i_\zeta}=h$,
i.e., means we get quasi-uniform meshes.
Using inequality $\mu \le 1$ and 
inequality $(a+b)^\gamma \leq 2^{\gamma-1}(a^\gamma+b^\gamma)$,
which can easily be shown 
since
the function $t^\gamma$ is convex 
in $(0,\infty)$,
we arrive at the estimates
\begin{align}\label{3.3}
h_{i_\zeta} & =C\frac{((n_{\zeta}+\nu)h)^{\frac{1}{\mu}}-(n_{\zeta} h)^{\frac{1}{\mu}}}{int(\nu^\frac{1}{\mu})} 
 \leq C\frac{C_\mu (n_{\zeta} h)^{\frac{1}{\mu}} +C_\mu (\nu h)^{\frac{1}{\mu}} -(n_{\zeta} h)^{\frac{1}{\mu}}}{int(\nu^\frac{1}{\mu})} \notag \\
& \leq  (C_{1,R,\mu,\nu}n_{\zeta})^{\frac{1}{\mu}} h h^{\frac{1}{\mu}-1} 
\leq  \frac{(C_{1,R,\mu,\nu}n_{\zeta})^{\frac{1}{\mu}}}{n_{\zeta}^{\frac{1-\mu}{\mu}}} h \Big(\big(n_{\zeta} h\big)^{\frac{1}{\mu}}\Big)^{1-\mu} \notag \\
& \leq (C_{2,\mu,\nu}n_{\zeta})^{\frac{1}{\mu}}h D^{1-\mu}_{(Z_{\zeta},P_s)},
\end{align}
which gives the right inequality in (\ref{3.2b}).
By the initial choice of $h_{i_\zeta}$, we have 
$h_{i_\zeta} = R_{\Omega_{i_{\zeta}}} / int(\nu^{-\mu})$.
Since $1> 1-\mu \geq 0$,  we can easily show that 
\begin{align}\label{3.3b}
	\frac{1}{int(\nu^\frac{1}{\mu})}R_{\Omega_{i_{\zeta}}}^{1-1+\mu} \geq C_m h,
\end{align}
with $C_m=\frac{1}{2}( (n_{\zeta}+\nu)^{\frac{1}{\mu}}-n_{\zeta}^{\frac{1}{\mu}})$. 
From the  choice grid sizes made above and (\ref{3.3b}), 
we can derive the left inequality in (\ref{3.2b}). 
\begin{remark}
It is possible to apply  other techniques, see for example \cite{LMMT:ApelSandingWhiteman:1996,LMMT:ApelMilde1996},
of constructing graded meshes, where 
we could prove optimal rates for the dG IgA method. 
We prefer the way that is described above for its simplicity and because it suits
to the spirit  of the dG IgA methodology.
\end{remark}

\subsection{Quasi-interpolant, error  estimates}
Next,  we study the { error estimates} of the method (\ref{3}). 
For the purposes of our analysis, we consider the enlarged  space 
\begin{align}\label{2.6a}
 W_h^{l,p}:=W^{1,2}(\Omega)\cap W^{l\geq 2,p}(\cal{T}_H(\Omega))+ \mathbb{B}_h(\cal{T}_H(\Omega)),
 \end{align}
 where $p\in (\max\{1,\frac{2d}{d+2(l-1)}\},2]$. 
 Let us mention that we allow different $l$ and $p$ in different
 subdomains $\Omega_i$. In particular, for subdomains ${\Omega}_i \cap U_s = \emptyset$, 
 we can set in  (\ref{2.6a}) $p=2$, for  subdomains ${\Omega}_i \cap U_s \neq \emptyset$, 
we set $1<p=\frac{2}{2-\lambda} < 2.$ 
 The space $W_h^{1,2}$ 
is equipped with the  broken dG-norm 
\begin{equation}\label{20}
 \|u\|^2_{dG(\Omega)} = \sum_{i=1}^N\Big(\alpha^{(i)}\|\nabla u^{(i)}\|^2_{L^2(\Omega_i)} +
                p_i(u^{(i)},u^{(i)}) \Big), {\ }u\in  W_h^{1,2}.
\end{equation}
Let $f\in W^{1,2}(\Omega)\cap W^{l\geq 2,p}(\cal{T}_H(\Omega))$ with 
$p\in (\max\{1,\frac{2d}{d+2(l-1)}\},2]$, then 
 we can construct a quasi-interpolant $\Pi_h f\in  \mathbb{B}_h(\mathcal{T}_{H})$ 
%
such that
$\Pi_h f=f$ for all $f\in \mathbb{B}_h(\mathcal{T}_{H})$.
We refer  to  \cite{LMMT:Schumaker_Bspline_book}, 
see also  \cite{LMMT:BazilevsBeiraoCottrellHughesSangalli:2006a},
for more details about the construction of $\Pi_h f$.
We have the following approximation estimate.
\begin{lemma}\label{lemma5.3}
Let $u\in W^{1,2}(\Omega)\cap W^{l,p}(\cal{T}_H(\Omega))$ with 
$p\in (\max\{1,\frac{2d}{d+2(l-1)}\},2]$
and $l\geq 2$,
 and let 
$E=\mathbf{\Phi}_i(\hat{E}), \hat{E}\in T^{(i)}_{h_i,\hat{\Omega}}$. 
Then, for  $0 \leq m \leq l \leq k+1$, 
  there exist an quasi-interpolant
  $\Pi_h u \in  \mathbb{B}_{h}(\cal{T}_H)$ and  
 constants 
 $C_i:=C_i\big(\max_{l_0 \leq l}(\|D^{l_0}\mathbf{\Phi}_i\|_{L^{\infty}(\Omega_i)})\big)$
 such that
 \begin{align}\label{3.5}
 \sum_{E\in T^{(i)}_{h_i,\Omega_i}} |u-\Pi_h u|^p_{W^{m,p}(E)} \leq C_i  \|u\|_{W^{l,p}(\Omega_i)} h_i^{p(l-m)}. 
\end{align}
 Furthermore, we have the following estimates in  the $\|.\|_{dG(\Omega)}$ norm
 \begin{subequations}\label{3.6}
 \begin{align}\label{3.6a}
  \|u-\Pi_h u\|_{dG(\Omega)} \leq &  \sum_{i=1}^N C_i\Big( h_i^{\delta (l,p,d)}\|u\|_{W^{l,p}(\Omega_i)}\Big)  +\\
\nonumber
 &\qquad \sum_{i=1}^N \sum_{F_{ij}\subset \partial \Omega_i}C_i\alpha^{(j)}\frac{h_i}{h_j}\Big( h_i^{\delta (l,p,d)}\|u\|_{W^{l,p}(\Omega_i)}\Big),\\
 \label{3.6b}
  \|u_h-\Pi_h u\|_{dG(\Omega)} \leq &  \|u-\Pi_h u\|_{dG(\Omega)} +\sum_{i=1}^N C_i h_i^{\delta (l,p,d)}\|u\|_{W^{l,p}(\Omega_i)},
 \end{align}
 \end{subequations}
 where $\delta (l,p,d)= l+(\frac{d}{2}-\frac{d}{p}-1)$. 
 \end{lemma}
\begin{proof}
	The proof is given in \cite{LMMT:LangerToulopoulos:2014a}.
	\hfill $\Square$
\end{proof}
\begin{remark}
If $h_i$ and $h_j$ are the grid sizes  of two adjacent
subdomains  $\Omega_i$ and $\Omega_j$, 
then relations (\ref{3.2}) immediately yield the 
two-side estimate
 $\sigma_{m,\zeta} \leq {h_i}/{h_j} \leq  \sigma_{M,\zeta}$,
where the positive constants $\sigma_{m,\zeta}$ and $\sigma_{M,\zeta}$ only depend  
on the quantities which specify the initial zone partition $Z_\zeta$. 
Hence,  in what follows, estimate (\ref{3.6a}) 
will be  used in the form $\|u-\Pi_h u\|_{dG(\Omega)} \leq   \sum_{i=1}^N C_i  h_i^{\delta (l,p,d)} $.
We mention that the analysis presented here
can easily be extended to non-matching grids, 
see \cite{LMMT:LangerToulopoulos:2014a}.
We also note that the meshes 
$T^{({i_\zeta})}_{h_{i_\zeta}}(\Omega_{i_\zeta})$ satisfy the Assumption~\ref{assume1}.
\end{remark}
%
We emphasize that Lemma~\ref{lemma5.3} provides
local estimates which hold in every subdomain $\Omega_i$. 
This help us to investigate
the accuracy of the method in every zone $Z_\zeta$  of $U_s$ separately.
We give an approximation estimate for the case where $u_r \in W^{l,2}(\Omega)$ with $l\geq k+1$, see (\ref{0.4}).
\par
   For all $\Omega_{i_\zeta}\in Z_{\zeta}$, 
 the local interpolation estimate
    (\ref{3.6a}) gives
     \begin{align}\label{3.7}
 	  \|u_s-\Pi_h u_s\|_{dG(U_s)} \leq \sum_{{i_\zeta}} h_{{i_\zeta}}^{\lambda} C_{{i_\zeta}},
     \end{align}
         since $u_s\in W^{l=2,p=\frac{2}{2-\lambda}}(\Omega)$. 
\begin{theorem}\label{Thrm1}
Let $Z_\zeta$ be a  zone partition of $\Omega$ with the properties 
listed in the previous section, and
let $T_{h_i}^{(i)}(\Omega_i)$ be the meshes of the subdomains as described in Section 3.1.  
Then, for the solution $u$ of (\ref{0.0a}), we have the 
error estimate
\begin{align}\label{3.8}
	\|u- u_h\|_{dG(\Omega)} \leq h^{r} C, \quad \text{with} \;
	r=\min\{k,{\lambda}/{\mu}\},
\end{align}
where the constant $C>0$ is determined by the quasi uniform mesh properties, see (\ref{3.2}), and 
the constants $C_i$ of Lemma \ref{lemma5.3}.
\end{theorem}
\begin{proof}
Let $\Pi_h u \in  \mathbb{B}_{h}(\cal{T}_H)$ be the quasi-interpolant of Lemma \ref{lemma5.3}.
Using the triangle  inequality, we obtain
\begin{align}\label{3.9}
 \|u-u_h\|_{dG(\Omega)} \leq \|u_h-\Pi_h u\|_{dG(\Omega)}  +	\|u-\Pi_h u\|_{dG(\Omega)}.
\end{align}
%
Moreover, representation (\ref{0.4}) yields
 \begin{align}\label{3.10}
 	\|u-\Pi_h u\|_{dG(\Omega)} \leq \|u_s-\Pi_h u_s\|_{dG(\Omega)} + \|u_r-\Pi_h u_r\|_{dG(\Omega)}.
 \end{align}
Using the fact that $u_r\in W^{l\geq k+1,2}(\Omega)$,  Lemma \ref{lemma5.3}, 
the mesh properties (\ref{3.2}) and
inequalities
$ 0< \mu \leq 1$, 
we have
\begin{align}\label{3.11}
	\|u_r-\Pi_h u_r\|_{dG(\Omega)} \leq C_1 h^{\frac{k}{\mu}}+C_2 h^k \leq C h^{k},
\end{align}
where the constants $C_1$ and $C_2$ 
are determined by the constants that appear in (\ref{3.2}) and (\ref{3.6a}).
Therefore, it remains to estimate the first term in (\ref{3.10}). 
By (\ref{3.2a}) and (\ref{3.7}),  
we obtain  the   estimate 
\begin{align}\label{3.12}
\|u_s-\Pi_h u_s\|_{dG(Z_0)} \leq C  h_{i_0}^{\lambda}\leq C h^{\frac{\lambda}{\mu}}
\end{align}
in $Z_0$, 
where the constant  $C$  
is determined by the constants  in (\ref{3.2}) and (\ref{3.6a}).
For the subdomains belonging to the remaining zones of $Z_{\zeta},{\ }\zeta \neq 0$, 
(\ref{3.2b}) and (\ref{3.7}) yield the estimates
%
\begin{align}\label{3.13}
\nonumber
 \|u_s-\Pi_h u_s\|_{dG(Z_\zeta)} 
&\leq C_{\zeta \neq 0} h_{i_\zeta}^{\lambda}
	\leq C_{\zeta \neq 0} \big( h D^{1-\mu}_{(Z_\zeta,P_s)}\big)^\lambda\\	  
&\leq	C_{\zeta \neq 0}\big(h h^{\frac{1-\mu}{\mu}}\big)^\lambda 
	\leq 	C_{\zeta \neq 0} h^\frac{\lambda}{\mu}.
\end{align}

Collecting  (\ref{3.11}), (\ref{3.12}) and (\ref{3.13}), 
we arrive at the interpolation error estimate
\begin{align}\label{3.14}
\|u-\Pi_h u\|_{dG(\Omega)} \leq C h^{r}, 
\end{align}
with $r=\min\{k,{\lambda}/{\mu}\}$.
Now, inserting 
estimate  (\ref{3.14}) into (\ref{3.9}) and recalling 
estimate (\ref{3.6b}),
we can easily derive the error estimate (\ref{3.8}). 
\hfill $\Square$
\end{proof}

\section{Numerical examples}
In this section, we present a series of numerical examples in order to confirm the theoretical
results and to assess the effectiveness of the proposed grading mesh technique.  
The first examples concern two-dimensional problems with boundary point singularities
and with highly discontinuous coefficients.
In the  last examples, we consider applications of the method to three-dimensional problems 
with an interior singularity and in domains with singular edges 
having $\omega= 3\pi/2$ interior angle. 
The numerical examples have been 
performed
in 
\gismo \footnote{\emph{Geometry + Simulation Modules}, http://www.gs.jku.at}.
%

\subsection{Implementation details}

The grading of the mesh is done in the parameter domain. The
underlying assumption is that the given parameterization of the domain
has uniform speed
along the patch. An ideal situation is to have an
arc-length parameterization. Nevertheless, a well-behaving
parameterization is one whose speed is within a constant factor of the
arc-length parameterization. This is a reasonable assumption, also
because CAD software typically try to adhere to such a requirement,
since it is desirable for CAD operations as well.

\par
For constructing the graded parameter mesh, we choose a number of
interior knots in each parametric direction and we place the knots
according to the grading parameter and the location of the singular
point in parameter space. In Fig.~\ref{Fig:GradedBasis}(a), we show the 
one-dimensional B-spline basis on a graded mesh, and similarly in
In Fig. \ref{Fig:GradedBasis}(b), we present the two- dimensional 
B-spline basis on  the corresponding graded mesh.
If the location of the singular point is given in physical
coordinates, we invert the point to parameter space with a Newton
iteration, to map it back to parameter space.
\begin{figure}
   \begin{subfigmatrix}{2}
   \subfigure[]{\includegraphics[scale=0.195]{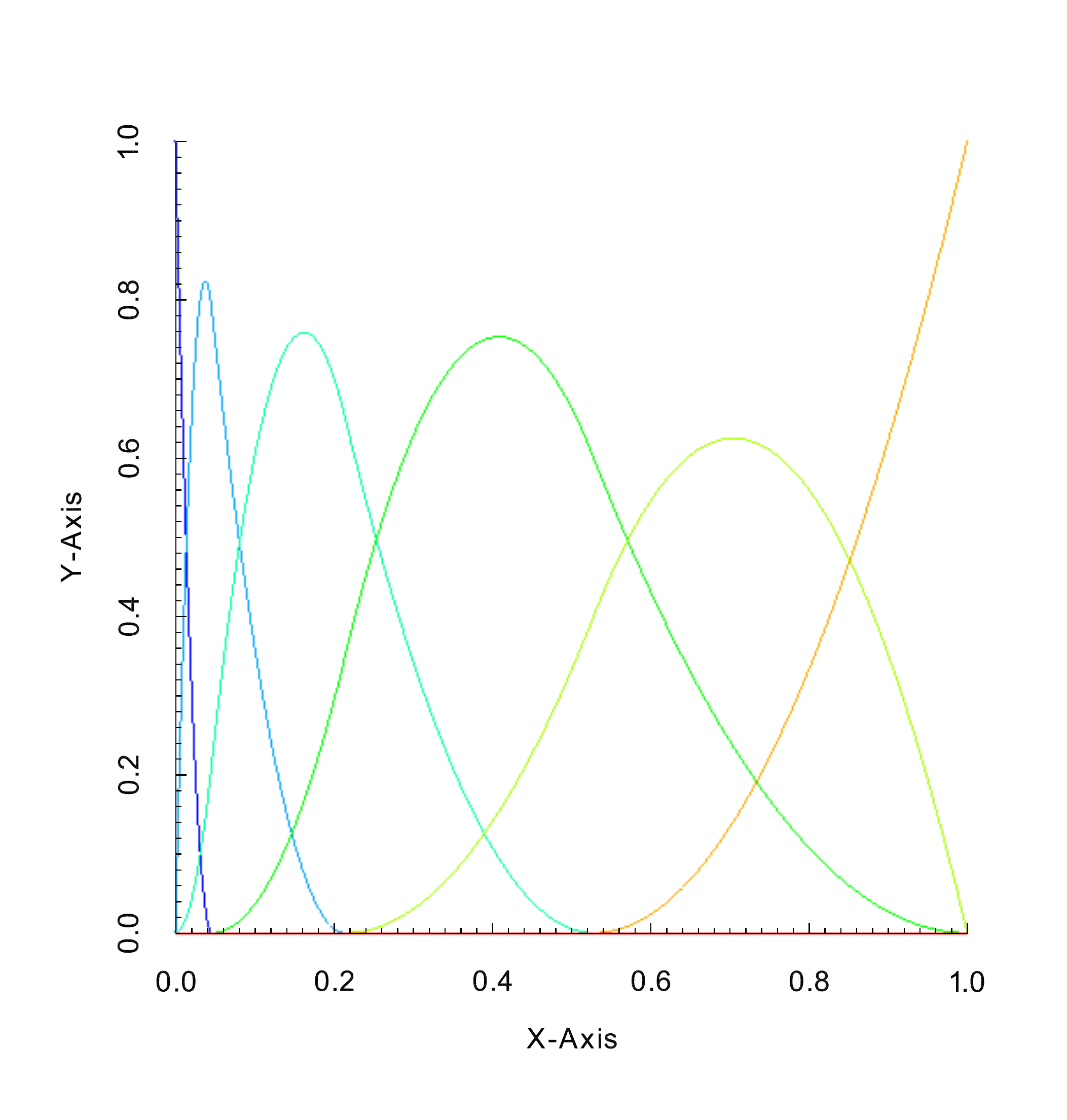}}
   \subfigure[]{\includegraphics[scale=0.195]{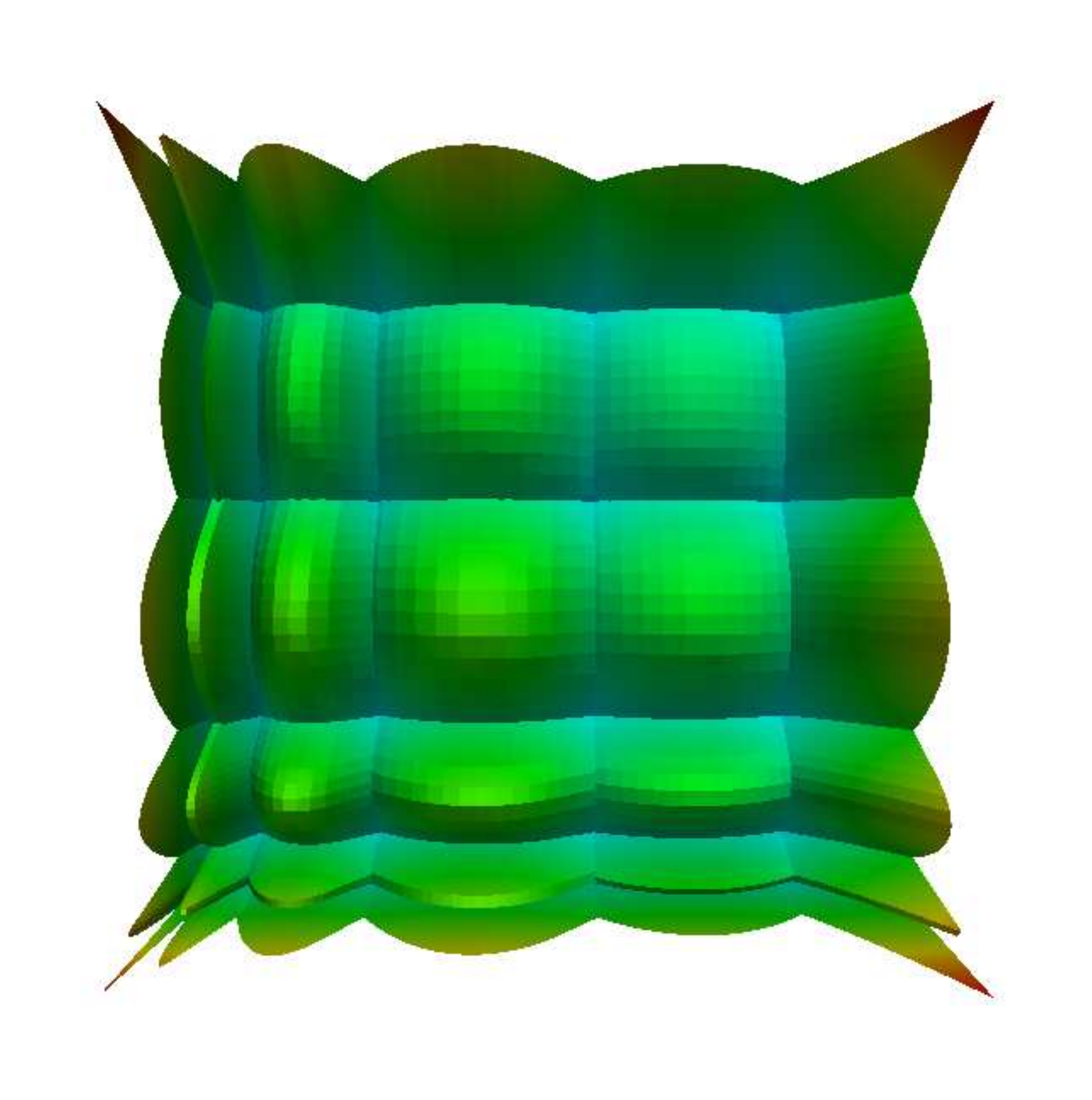}}
 \end{subfigmatrix}
   \caption{Basis functions on the graded mesh $T^{(i_{\zeta})}_{h_{i_{\zeta}},\widehat{\Omega}}$:  
         (a) The 1d bases on   $\mu=0.6$ grading, (b) The 2D bases on $\mu = 0.6$ grading.}
         \label{Fig:GradedBasis}
 \end{figure}
Under the assumption of a well-behaved B-spline geometry map, the
parameter mesh is transformed from
$T^{(i_{\zeta})}_{h_{i_{\zeta}},\widehat{\Omega}}$ to
$T^{(i)}_{h_i,\Omega_i}$ in such a way that the size of the physical
elements are proportional to their (pre-image) parametric
elements. This property ensures that our theoretical analysis applies
in the experiments that we conducted.
\par
For efficiency reasons, the mesh $\cal{T}_H(\Omega)$ is created by the
grading function with the same number of knots as an equivalent
uniform mesh with grid size $h_i$ for each subdomain $\Omega_i$, but pulled towards the
singularity $P_s$ using the grading parameter $\mu$.  This strategy
satisfies Assumption~\ref{assume1} with a minimal number of knots.
This approach also mimics the zones construction introduced in
Subsection \ref{apriorimesh}, since it corresponds to a zone partition
that shrinks towards $P_s$ at every refinement step.
\par
In our experiments, we consider a mapping $\mathbf{\Phi}_i$ produced
on an initial knot vector $\mathbf{\Xi}_i^d$, which exactly represents
the subdomain $\Omega_i$, as the isogeometric paradigm suggests. The
knots are relocated during the grading procedure but without changing
the shape or the parameterization of the subdomains $\Omega_i$.  If
needed, the original coarse knots are inserted in the discretization
basis such that the exact representation of the original shape is feasible.
Nevertheless, in our implementation, we have the freedom to use a
different sequence of knots for the discretization space without
refining the initial geometry in this basis. 
%
%
\subsection{Numerical Examples for Two dimensional}
\subsubsection{Heart shaped domain}
   To illustrate the efficiency of the proposed 
    mesh grading methodology and to validate the estimates of 
    Section~\ref{sec:ProblemDescriptionAndDiscretization}, we consider the problem 
    \eqref{0.0} in a curved domain (heart shape) having a singular point $P_s$ (re-entrant corner) 
    with internal angle $\omega = 3\pi /2$, see  Fig.~\ref{fig:sec:heart}(a). 
 %
 %
    \begin{figure}[bth!]
      \begin{subfigmatrix}{3}
      \subfigure[]{\includegraphics[width=0.3\textwidth]{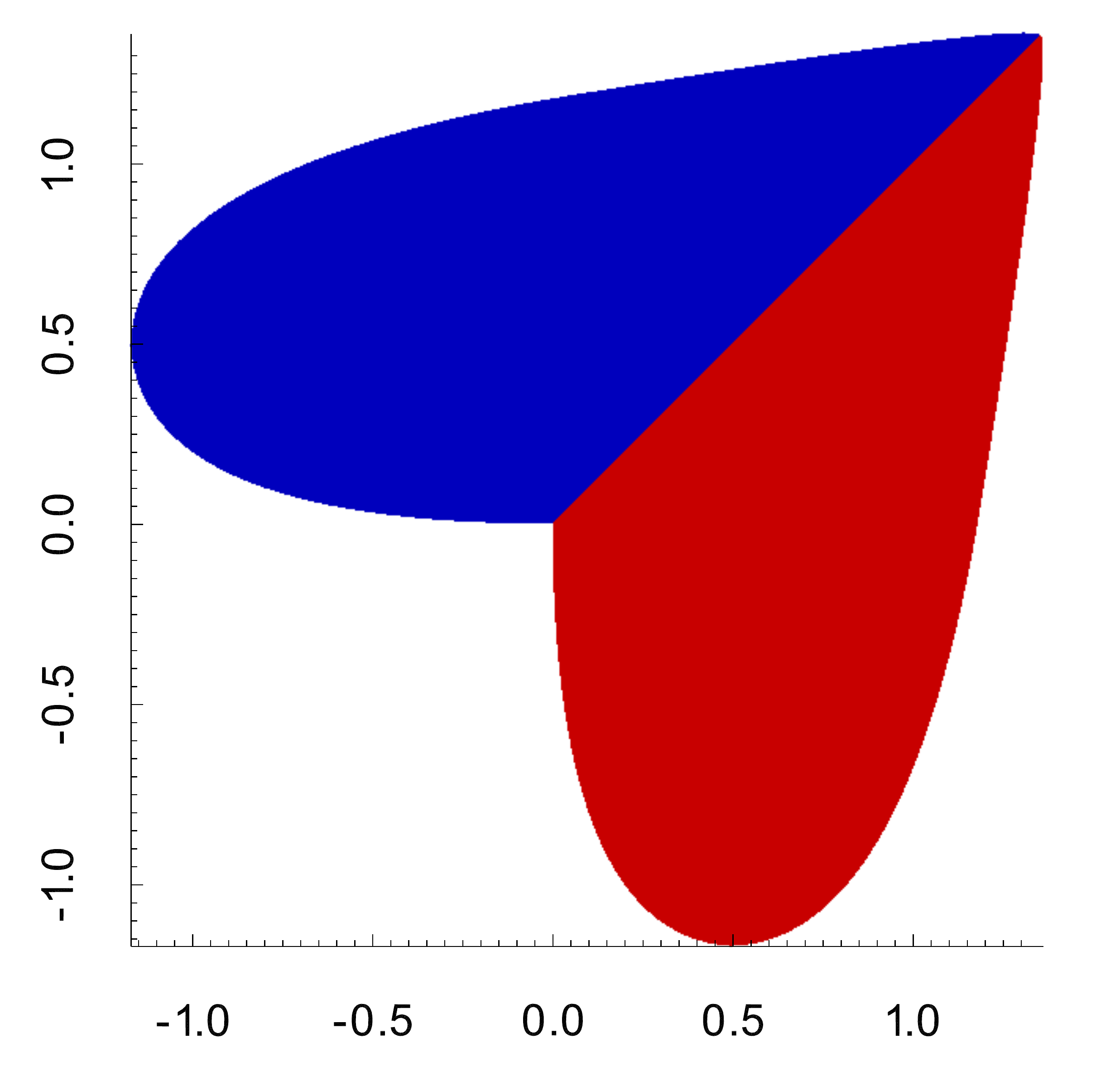}}
      \subfigure[]{\includegraphics[width=0.3\textwidth]{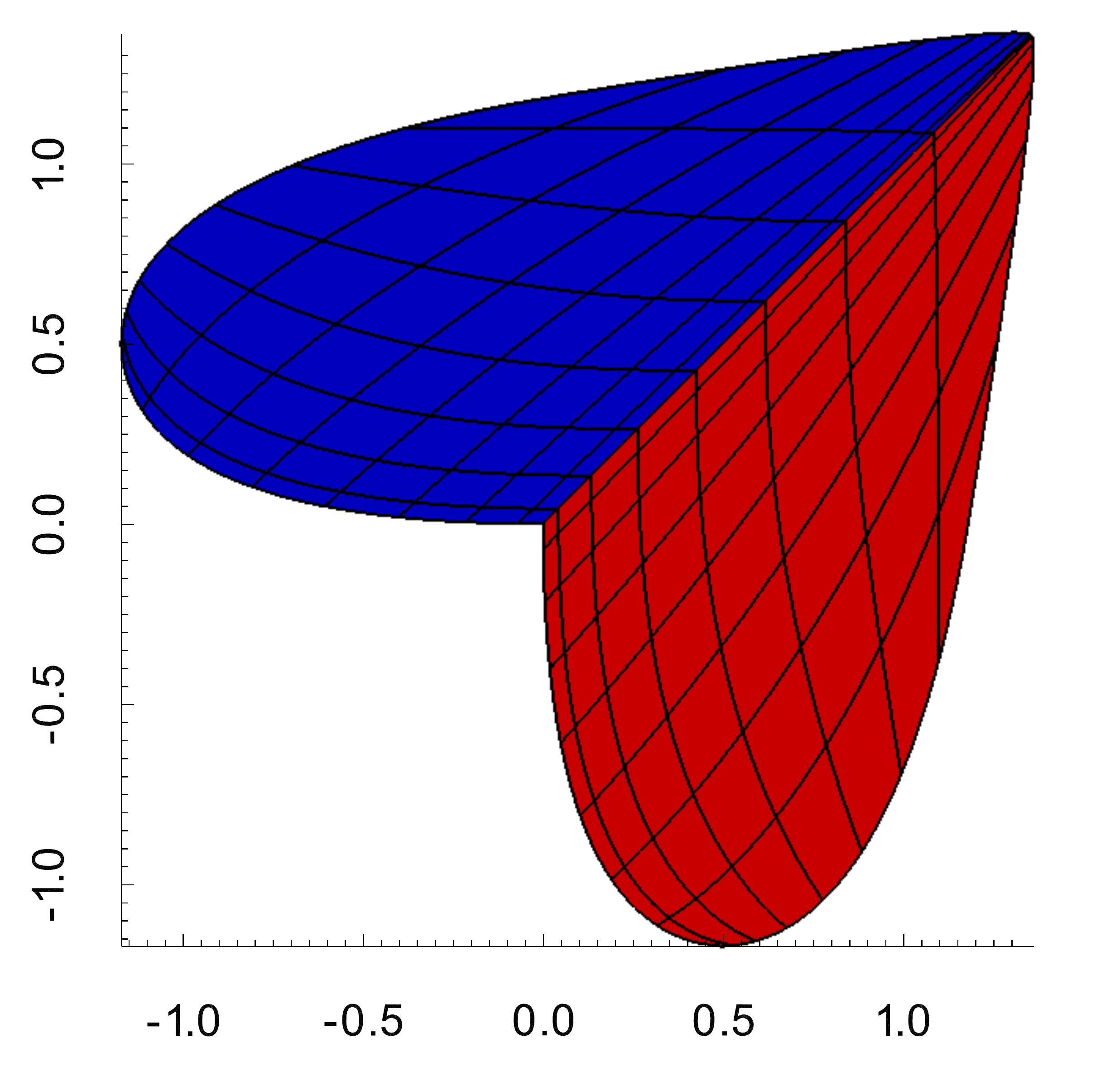}}
       \subfigure[]{ \includegraphics[width=0.3\textwidth]{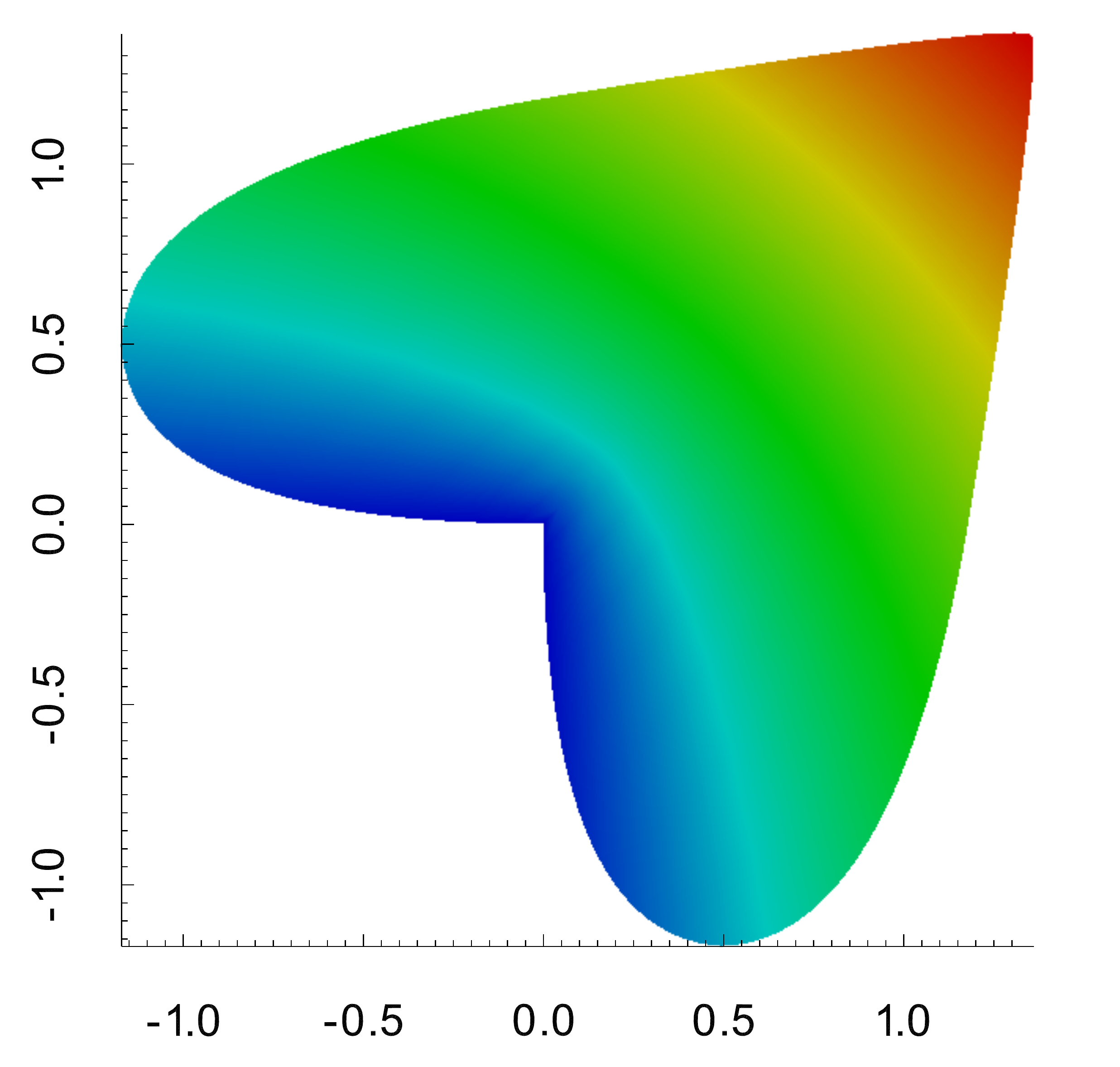}}
     \end{subfigmatrix}
	    \caption{Heart shape problem: (a) the computational domain and the subdomains, 
                                          (b) the graded meshes of the two subdomains,
                                          (c) the contours of $u_h$ solution. }
    \label{fig:sec:heart}
    \end{figure}
    
\begin{table}[htb!]
\centering %
\begin{tabular}{|l|l|l|l|} 
\hline 
 $j_1$&  $j_2$       &$B^{1}_{j}(j_1,j_2)$  &$B^{2}_{j}(j_1,j_2)$ \\ \hline
  $1$&  $1$         &$(0.00, 0.00)$    &$(0.00,0.00)$\\
  $2$&  $1$         &$(0.49,0.49) $    &$(0.49,0.49)$\\
  $3$&  $1$         &$(0.97,0.97) $    &$(0.97,0.97)$\\
  $1$&  $2$         &$(0.00,-0.81)$    &$(-0.81,0.00)$\\
  $2$&  $2$         &$(0.46,-0.16)$    &$(-0.16,0.46)$\\
  $3$&  $2$         &$(1.00,0.94) $    &$(0.94, 1.00)$\\
  $1$&  $3$         &$(0.35,-0.84)$    &$(-0.84,0.35)$\\
  $2$&  $3$         &$(0.71,-0.84)$    &$(-0.84,0.71)$\\
  $3$&  $3$         &$(0.85,0.042)$    &$(0.04 ,0.85)$\\ \hline
\end{tabular}
\caption{The control points for the two B-spline surfaces each with degree $k=2$ depicted 
         in Figure~\ref{fig:sec:heart}}
\label{tab:sec:heart:controlpoints}
\end{table}

\begin{table}[htb!]
\centering %
\begin{tabular}{|c|l|l|l|l|} 
\hline 
          &\multicolumn{2}{|l|}{without grading }
          &\multicolumn{2}{|l|}{with grading } \\ [0.5ex] \hline 
$h / 2^s$    &$k=1$& $k=2$&$\begin{matrix}
                                      k=1,\\ \mu=0.6
                                  \end{matrix}$  
                                &$\begin{matrix}
                                      k=2, \\ \mu=0.3
                                  \end{matrix}$ \quad
                                \\ \hline 
        \multicolumn{5}{|c|} {Convergence rates } \\ [0.5ex]  \hline
 $s=0$  &-        & -       & -        &-         \\ 
 $s=1$  &0.671469 &0.68221  &0.843026  &1.49519   {\ }   \\ 
 $s=2$  &0.678694 &0.67322  &0.894636  &1.85785   {\ }   \\ 
 $s=3$  &0.677558 &0.669385 &0.921219  &2.02913   {\ }   \\ 
 $s=4$  &0.675018 &0.667797 &0.938709  &2.02562   {\ }   \\ 
 $s=5$  &0.672622 &0.667156 &0.951475  &2.00987   {\ }   \\ \hline
\end{tabular}
\caption{Heart shape problem: The convergence rates in the dG-norm  $\|.\|_{dg(\Omega)}$  
	  with and without grading.}
\label{tab:sec:Heart}
\end{table}

The computational domain $\Omega$ consists of two subdomains shown 
    in Fig.~\ref{fig:sec:heart}(a), where the corresponding
    knot vectors are  $\mathbf{\Xi}_i^2=(\Xi_{i}^1,\Xi_{i}^2),i=1,2$ with $\Xi_{i}^1=\Xi_{i}^2= \{0, 0, 0, 1, 1, 1\}$ 
    and are parametrized by $k=2$  B-spline  basis with the control points given in 
    Table~\ref{tab:sec:heart:controlpoints}.    
    The exact solution is given by $u = r^\frac{\pi}{\omega}\sin( \theta \pi / \omega)$
    with $f$ and $u_D$ in \eqref{0.0}
    are specified by the exact solution. We set $\alpha=1$ in the entire $\Omega$.  
      Note that 
    $u \in W^{1.5,2}(\Omega)$.
    The problem has been solved using first ($k=1$) and second ($k=2$) order 
    B-spline  spaces with grading parameter $\mu=0.6$ and $\mu=0.3$, respectively, 
    see Fig.~\ref{fig:sec:heart}(b). We plot the contours of the solution $u_h$ 
    computed using $k=2$  B-splines in Figure~\ref{fig:sec:heart}(c).
    In Table~\ref{tab:sec:Heart}, we display the convergence rate of the error. 
    In the left column (without grading), we present the rates 
    using quasi-uniform meshes.   
    In the right column of the table, we show the rates in the case of using graded meshes. As the theory predicts,  
    the convergence rates in left columns of both cases $k=1$ and $k=2$ are mainly determined by the regularity of the solution.  
    One the other hand, the rates which correspond to the graded meshes tend to be optimal
    with respect to the  B-spline  degree $k$. 
    This shows the adequacy of the proposed graded mesh for solving this problem. 
%
%
%
\subsubsection{Kellogg's Problem.}
It is known that the solutions of problem (\ref{0.0}) with rough diffusion coefficients
may not be very smooth. Thus, standard numerical method can not provide 
an (optimal) accurate approximation 
\cite{LMMT:RICHARD_FALKAND_OSBORN_MixedEll_1994}. 
We examine such  a case by solving the so-called Kellogg test problem 
\cite{LMMT:Kellog_DiscDifCoef_1975}.
We consider the computational domain $\Omega=(-1,1)^2$. 
The  diffusion coefficient $\alpha$ in (\ref{0})
is supposed to be piecewise constant taking the same value, say $\alpha:=\alpha_{13}$,  
in the first and third quadrants, and, similarly,  
for the second and   fourth  quadrants,
$\alpha:=\alpha_{24}$.
\begin{figure}[thb!]
          \begin{subfigmatrix}{3}
	   \subfigure[]{  \includegraphics[width=0.3\textwidth, height=0.22\textheight]{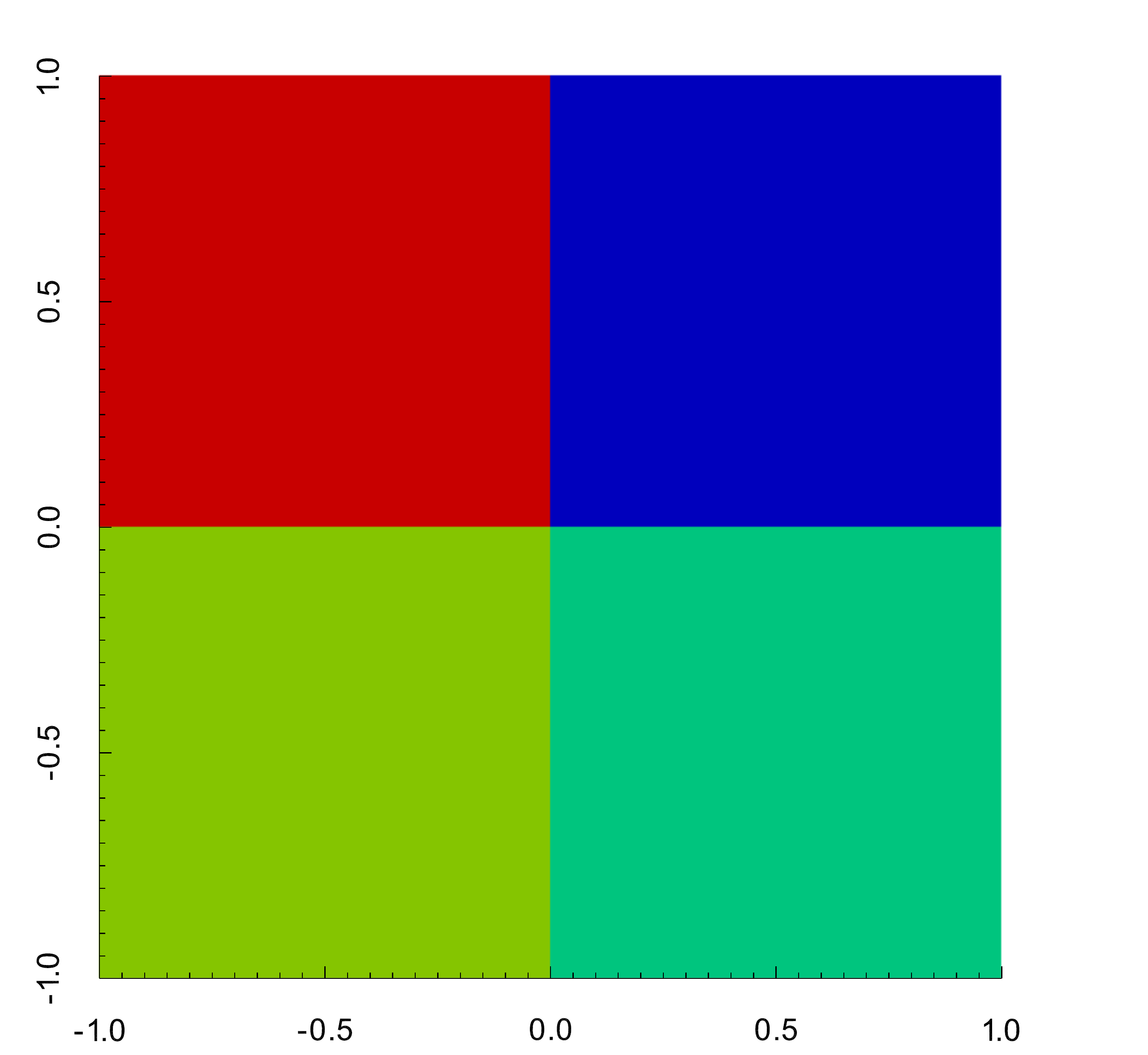}}
	    \subfigure[]{ \includegraphics[width=0.3\textwidth, height=0.22\textheight]{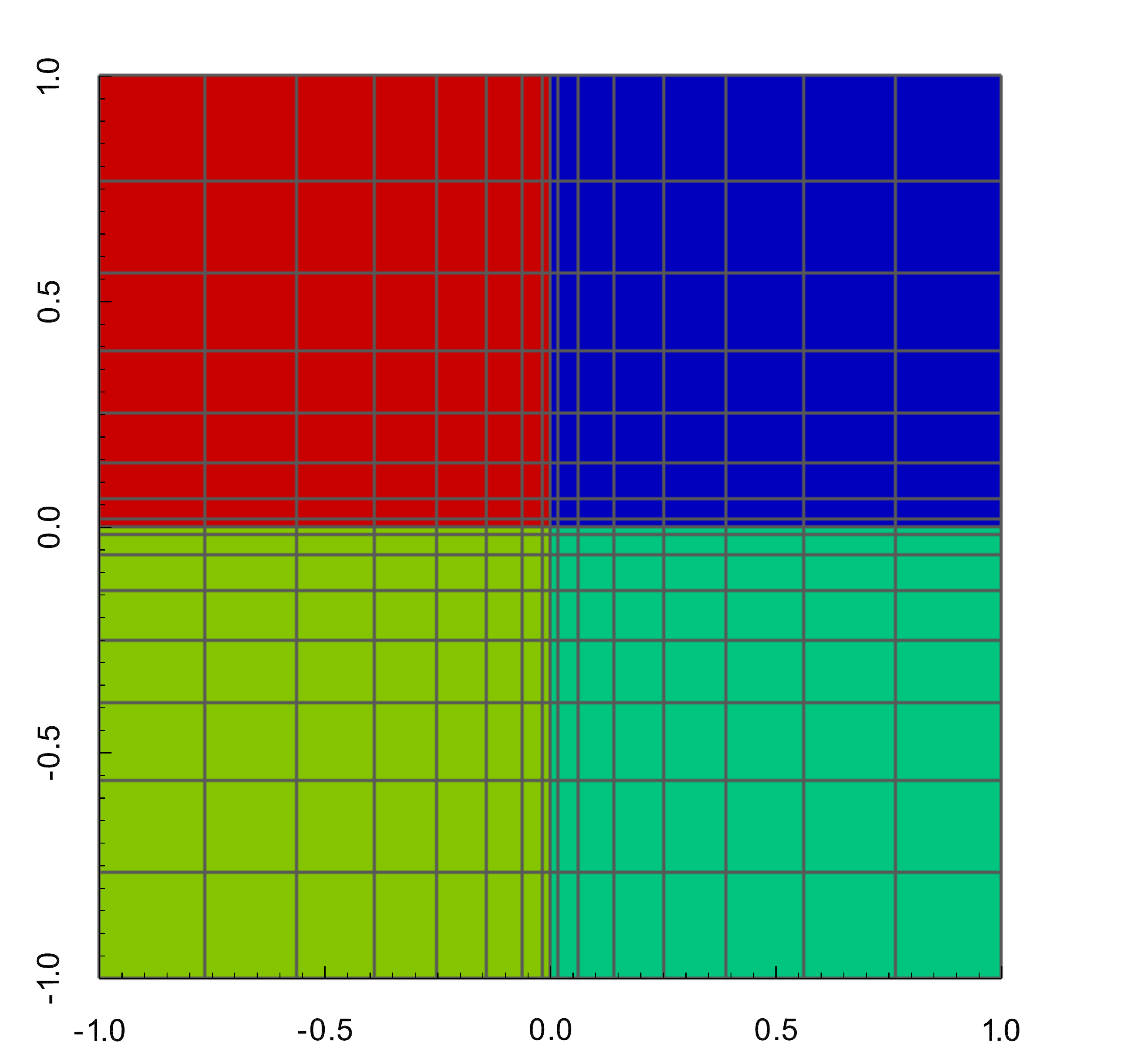}}
	   \subfigure[]{  \includegraphics[width=0.3\textwidth, height=0.22\textheight]{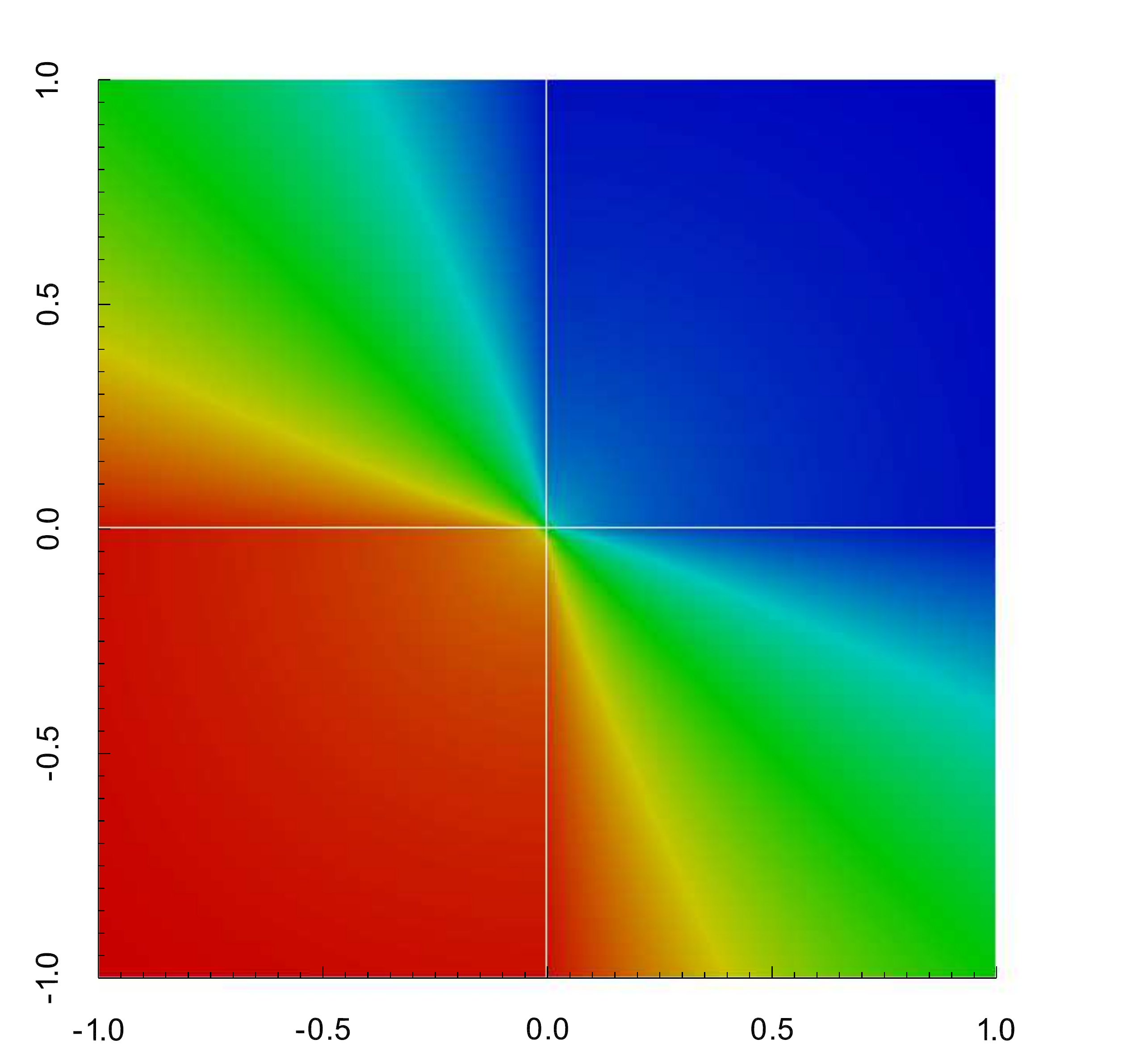}}
	   \end{subfigmatrix}
	    \caption{Kellogg's test problem: (a) The computational domain and the four subdomains, 
				     (b)  the graded meshes of the four subdomains,
				     (c)  the contours of $u_h$ solution.}
	      \label{Fig_Kellog_test}
\end{figure}
This choice of the diffusion coefficient leads to the subdivision of
$\Omega$ in the four subdomains as it is shown in Fig.~\ref{Fig_Kellog_test}(a). 
The exact solution of the problem  for $f=0$ is given in polar coordinates
         by $u(r,\theta) = r^{\lambda}\varphi(\theta)$, where 
\begin{align*}
          \varphi(\theta) = \begin{cases}
                         \cos( (\pi /2 -\sigma)\lambda) \cos((\theta-\pi/2+\rho)\lambda), & \text{if }  0\leq   \theta < \pi/ 2, \\
                         \cos(\rho \lambda) \cos((\theta-\pi+\sigma)\lambda),            & \text{if } \pi/2 \leq  \theta < \pi, \\
                         \cos(\sigma \lambda) \cos((\theta-\pi-\rho)\lambda),            & \text{if }  \pi\leq \theta < 3\pi /2   , \\
                         \cos( (\pi/2 -\rho)\lambda) \cos((\theta-3\pi/2-\sigma)\lambda),& \text{if } 3 \pi /2 \leq \theta \leq 2\pi , \\
                         \end{cases}
\end{align*}
where the numbers $\lambda, \rho, \sigma$ satisfy the nonlinear relations 
         \begin{align*}
          \begin{cases}
                 \mathcal{R}:=\frac{\alpha_{13}}{\alpha_{24}} = -\tan((\pi/2 -\sigma)\lambda) \cot(\rho\lambda),\\
                 \frac{1}{\cal{R}}=-\tan(\rho\lambda)\cot(\sigma\lambda),\\
                 \cal{R}=-\tan(\sigma \lambda)\cot((\pi/2-\rho)\lambda),\\
                 0<\lambda<2,\\
                 \max\{0,\pi\lambda-\pi\}<2\lambda\rho<\min\{\pi\lambda,\pi\},\\
                 \max\{0,\pi- \pi\lambda\}<-2\lambda\sigma<\min\{\pi,2\pi-\lambda\pi\}.
          \end{cases}
\end{align*}

          For $\lambda=0.4$, the solution 
	  $u\in W^{1.4,2}(\Omega)$,
	  and
          has discontinuous derivatives 
	  across
	  the interfaces. 
	  On the other hand, 
	  $u\in W^{2,1.25}(\Omega)$,
          and the estimates presented in Section~3.3 can be applied. 
          We solved the problem using B-spline spaces with degrees
          $k=1$ and $k=2$ on uniform meshes. We performed again
          the test using graded meshes with grading parameter chosen
          such that $\lambda / \mu =k $, see (\ref{3.8}).  In
          Fig.~\ref{Fig_Kellog_test}(b), we can see the graded meshes
          of the subdomains. 
	  Fig.~\ref{Fig_Kellog_test}(c) shows the 
          plot of the contours of the dG solution $u_h$ computed for degree $k=1$
          B-splines.  In Table \ref{Table_Kellog_test}, we display the
          convergence rates of the solution.  
	  We observe that, in the case of  uniform meshes,
	  the experimental order of convergence of the method
          is $0.4$ which is determined by  the regularity of the solution.
          Conversely, the rates in the right columns which correspond
          to the results using mesh grading tend to be optimal with
          respect the order of the B-spline space.
%
\begin{table}[hutb!]
    \centering %
      \begin{tabular}{|c|l|l|l|l|} 
	\hline 
		&\multicolumn{2}{|l|}{without grading }
		&\multicolumn{2}{|l|}{with grading } \\ [0.5ex] \hline 
      $h / 2^s$    &$k=1$& $k=2$&$\begin{matrix}
					    k=1,\\ \mu=0.40
					\end{matrix}$  
				      &$\begin{matrix}
					    k=2, \\ \mu=0.20
					\end{matrix}$ \quad
				      \\ \hline 
	      \multicolumn{5}{|c|} {Convergence rates } \\ [0.5ex]  \hline
      $s=0$  &-        & -         & -        &-         \\ 
      $s=1$   &0.655814&  0.591165&0.830217  &0.477657  {\ }   \\ 
      $s=2$   &0.354865&  0.355586&0.858329  &1.15442  {\ }   \\ 
      $s=3$   &0.368103&  0.378796&0.879976  &1.78696   {\ }   \\ 
      $s=4$   &0.378375& 0.385672 &0.895984  &1.84425   {\ }   \\ 
      $s=5$   &0.385464&  0.390348&0.906179  &1.95223   {\ }   \\ \hline
      \end{tabular}
      \caption{Kellogg's test: The convergence rates in the dG-norm  $\|.\|_{dg(\Omega)}$.}
      \label{Table_Kellog_test}
\end{table}
A glimpse of the discrete solution $u_h$ is given in Fig.~\ref{Fig:Kelloggs}.
\begin{figure}
    \centering
    \includegraphics[scale=0.195]{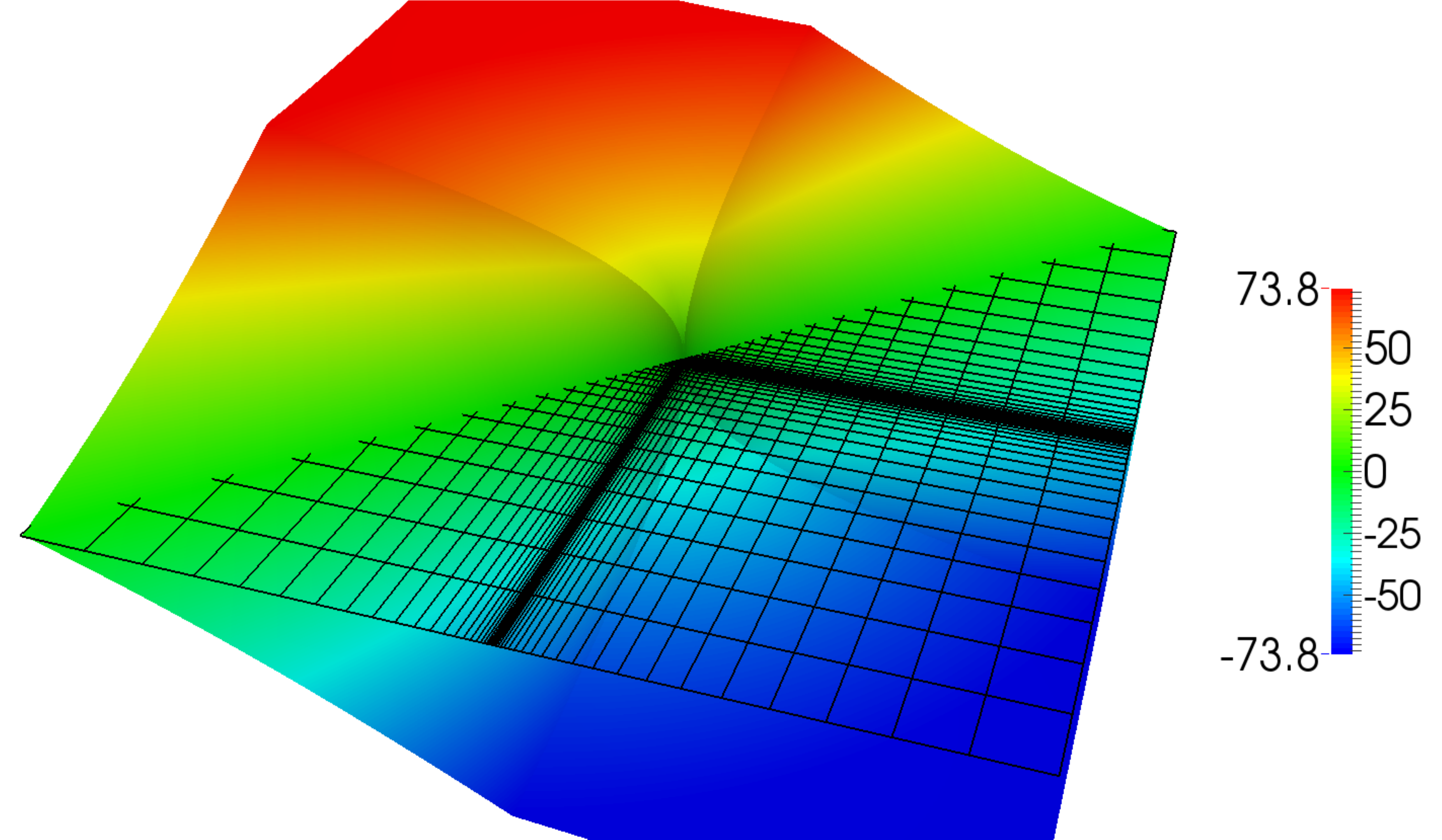}
\caption{Discrete solution of Kellogg's problem plotted over the graded mesh.}
\label{Fig:Kelloggs}
\end{figure}
%
%
\subsection{Three dimensional examples}

 \subsubsection{Cube with interior point singularity.}
   This test case is inspired  by \cite{LMMT:LDG_pLaplace_IT_2014}. 
  The computational domain is $\Omega=(-1,1)^{3}$ which is decomposed into 8 subdomains, 
  see  Fig.~\ref{fig:sec:cube}(a). 
  We choose the diffusion coefficient $\alpha = 1$ in the whole computational domain $\Omega$.
      The  solution of the problem has a singular point at the origin of the axis and is given by
      $u(x)=|x|^{\lambda}$ with $\lambda=0.85$.  It is easy to show that $u\in W^{l=2,2}(\Omega)$. 
      We solved the problem using $k=1, {\ }k=2$ and $k=3$ B-spline spaces on 
       quasi-uniform meshes. 
      In  Fig.~\ref{fig:sec:cube}(b), we plot the contours of  solution $u_h$ computed by $k=1$ B-spline space.  
       The convergence rates of the error corresponding to non-graded meshes 
       are shown in left columns of Table~\ref{tab:sec:cube}.
       
      \begin{figure}[bth!]
       \begin{subfigmatrix}{3}
       \subfigure[]{\includegraphics[width=0.32\textwidth]{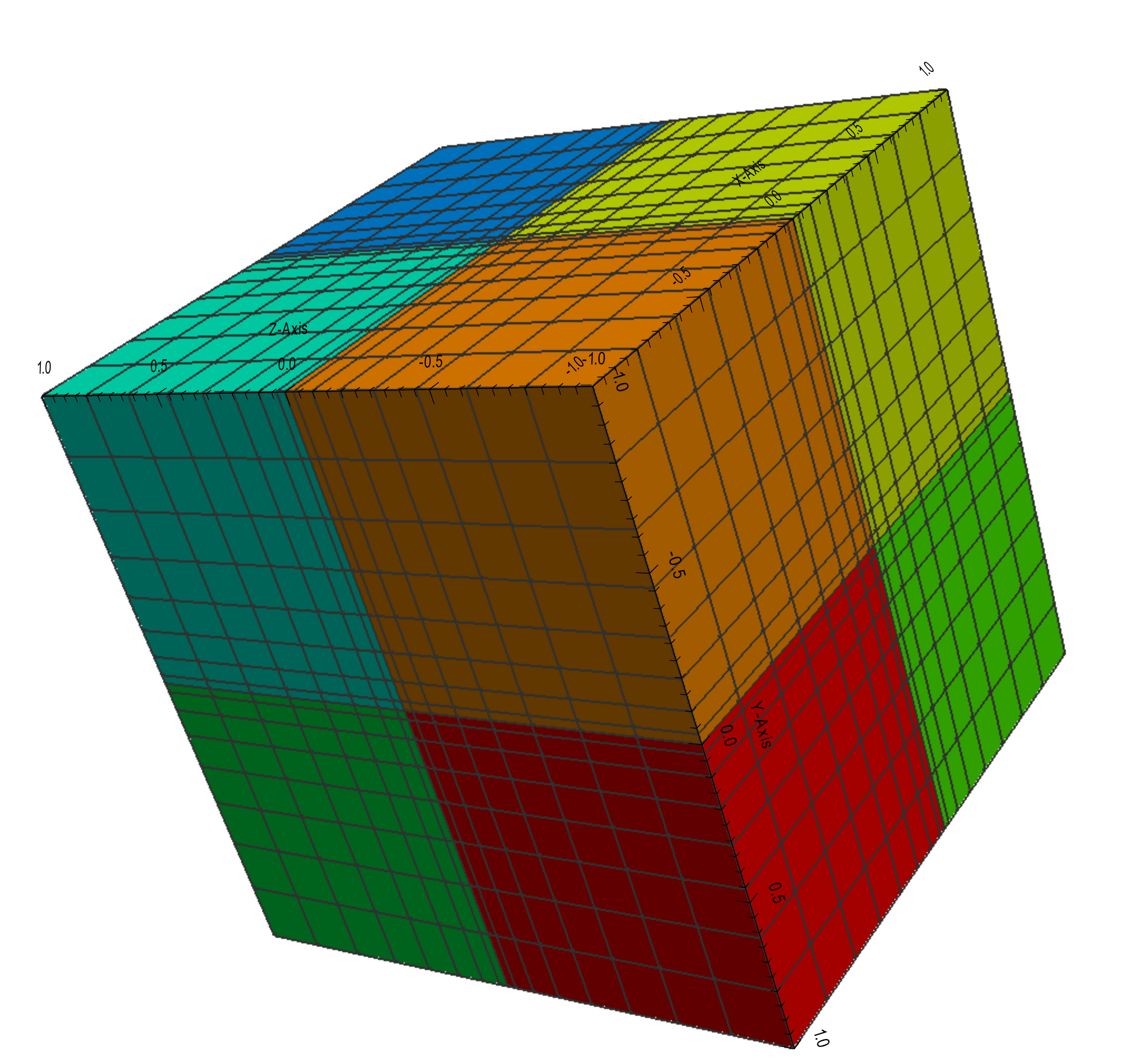}}
       \subfigure[]{\includegraphics[width=0.32\textwidth]{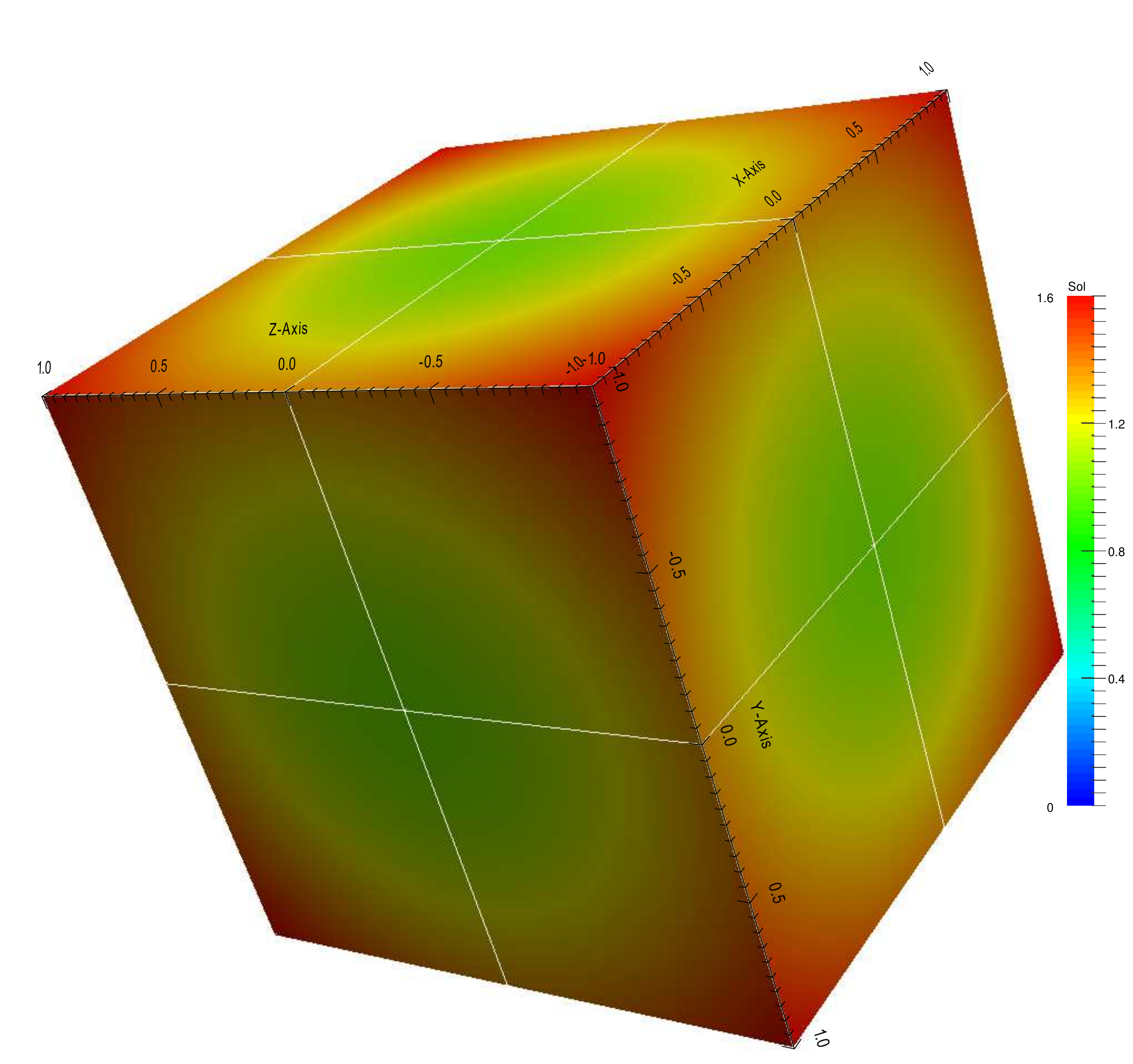}}
         \subfigure[]{\includegraphics[width=0.32\textwidth]{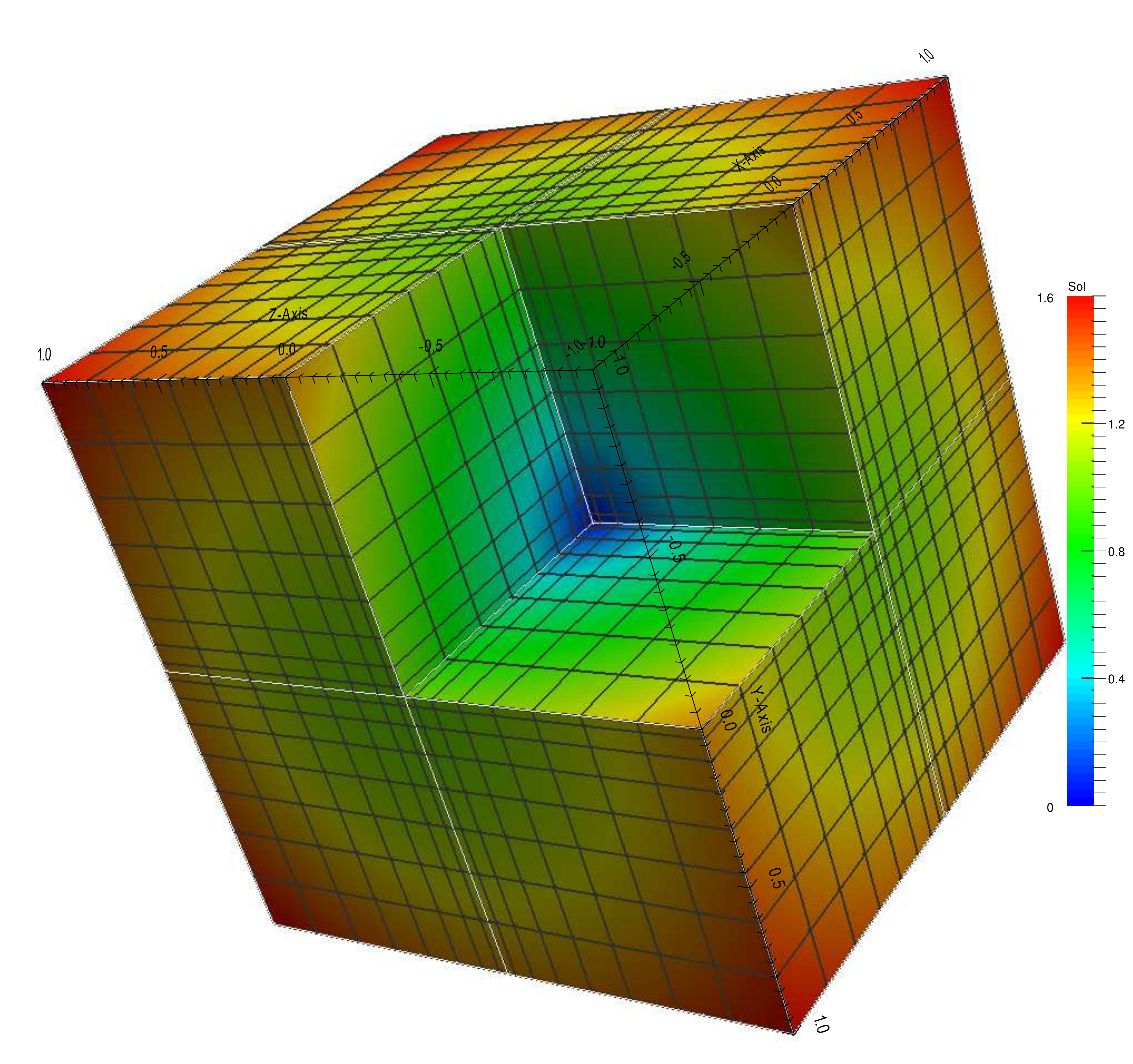}}
       \end{subfigmatrix}
	 \caption{Cube with interior singularity: (a) the decomposition of $\Omega$ into 8 subdomains 
	                                              with the graded meshes of the subdomain,
	                                         (b) the contours  of the solution $u_h$, 
	                                         (c) the variance of the $u_h$ contours around the singular point.}
      \label{fig:sec:cube}
      \end{figure}
      
       The rates are optimal for $k=1$ B-spline  space and sub-optimal for the two other B-spline  spaces, as 
       it was expected according to the regularity of the solution $u$. Note that the rates 
       presented on the left columns in Table \ref{tab:sec:cube} are in agreement with the estimate 
       given in (\ref{3.6a}). We have performed again the test 
       using grading meshes for the last two B-spline  spaces. The grading parameter $\mu$ has been chosen to be 
       $\delta(l=2,p=2,d=3) / \mu = k$, see Lemma \ref{lemma5.3} and (\ref{3.8}). 
       In  Fig.~\ref{fig:sec:cube}(c), the variate of the $u_h$ contours around the singular point 
       is  shown. The rates obtained on graded meshes  are displayed in right columns of Table \ref{tab:sec:cube}.

%

       \begin{table}[hutb!]
	 \centering %
	  \begin{tabular}{|c|l|l|l|l|l|l|} 
	    \hline 
		    &\multicolumn{3}{|l|}{without grading }
		    &\multicolumn{3}{|l|}{with grading } \\ [0.5ex] \hline 
	  $h /2^s$    &$k=1$& $k=2$&$k=3$ &$\begin{matrix}
						  k=1,\\ \mu= 1.0
					          \end{matrix}$  
					          &$\begin{matrix}
						  k=2, \\ \mu=0.6
					          \end{matrix}$ 
					          &$\begin{matrix}
						  k=3, \\ \mu=0.4
					          \end{matrix}$ \quad
					  \\ \hline 
		  \multicolumn{7}{|c|} {Convergence rates } \\ [0.5ex]  \hline
	  $s=0$  &-     & -     &-      & -     &-       &-     \\ 
	  $s=1$  &0.593 &1.066  &0.687  &0.593  &1.393   &0.791  {\ }   \\ 
	  $s=2$  &0.839 &1.306  &1.234  &0.839  &1.766   &1.870  {\ }   \\ 
	  $s=3$  &0.917 &1.340  &1.343  &0.917  &1.928   &2.942  {\ }   \\ 
	  $s=4$  &0.953 &1.346  &1.350  &0.953  &1.959   &3.080  {\ }   \\ 
	  $s=5$  &0.972 &1.348  &1.350  &0.972  &1.974   &3.066  {\ }   \\ \hline
	  \end{tabular}
	  \caption{Cube with interior singularity: The convergence rate of the error on uniform and graded meshes.}
		   \label{tab:sec:cube}
	  \end{table}

       We can observe that the rates approach the optimal rate for both high-order  B-spline  spaces. 
       This numerical example
       demonstrates that the dG IgA method applied on the proposed graded  meshes  can exhibit 
       optimal convergence rates for interior singularity 
       type  problems as well. 

    \subsubsection{Three-dimensional L-shape domain.}
Now the computational domain $\Omega$ has 3d L-shape form and is 
given by
    $\left( (-1, 1)^{2} \setminus (-1, 0)^{2} \right) \times [0,1]$.
    Even though the ''L-shape`` example has been mostly studied in the literature in its two-dimensional set up,
    (see for example anisotropic 2d meshes for IgA discretizations in \cite{LMMT:SangalliaNurbs2012}), 
    we believe that it  is an interesting test case, because we will see  that the graded mesh of the plane can be 
    prolonged  in a direction perpendicular to the singular edge for treating the boundary singularities.
    Note that in this three dimensional setting, the domain includes both corner and edge singularities, 
    see Fig. \ref{fig:sec:lshape3d}(a).  
    
    \begin{figure}[bth!]
      \begin{subfigmatrix}{3}
    \subfigure[]{ \includegraphics[width=0.32\textwidth, height=0.22\textheight]{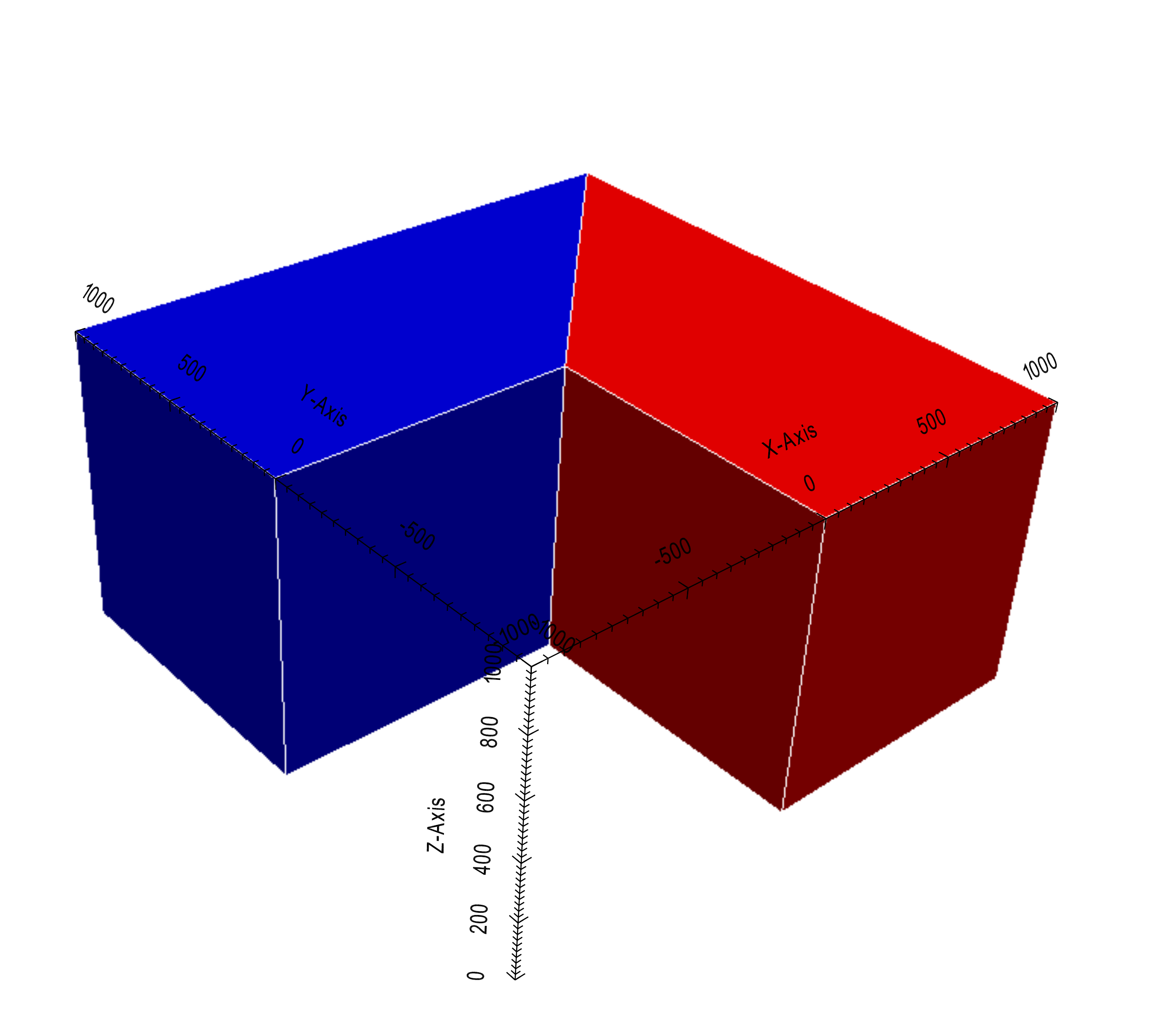}}
    \subfigure[]{ \includegraphics[width=0.32\textwidth, height=0.22\textheight]{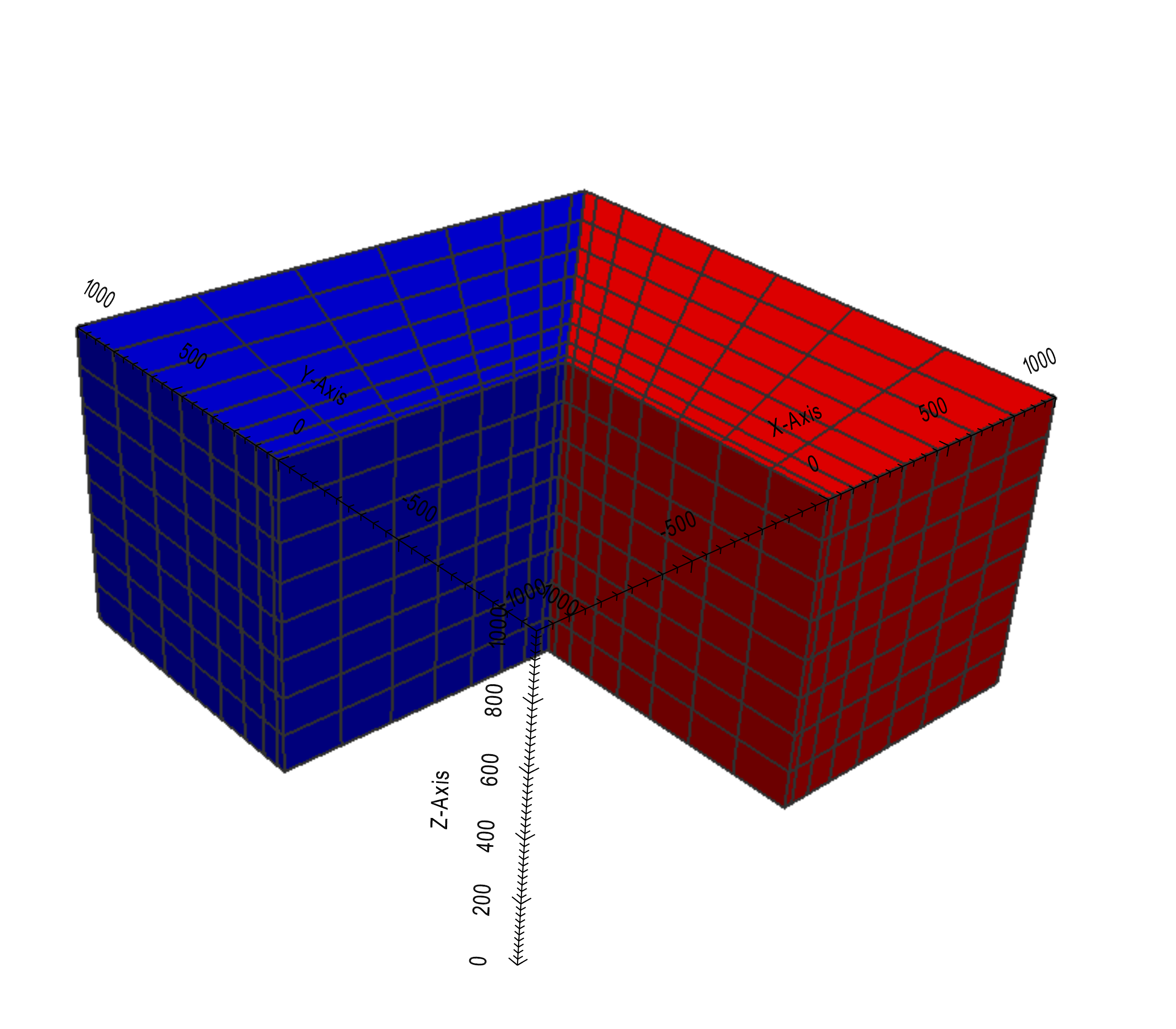}}
     \subfigure[]{\includegraphics[width=0.32\textwidth, height=0.22\textheight]{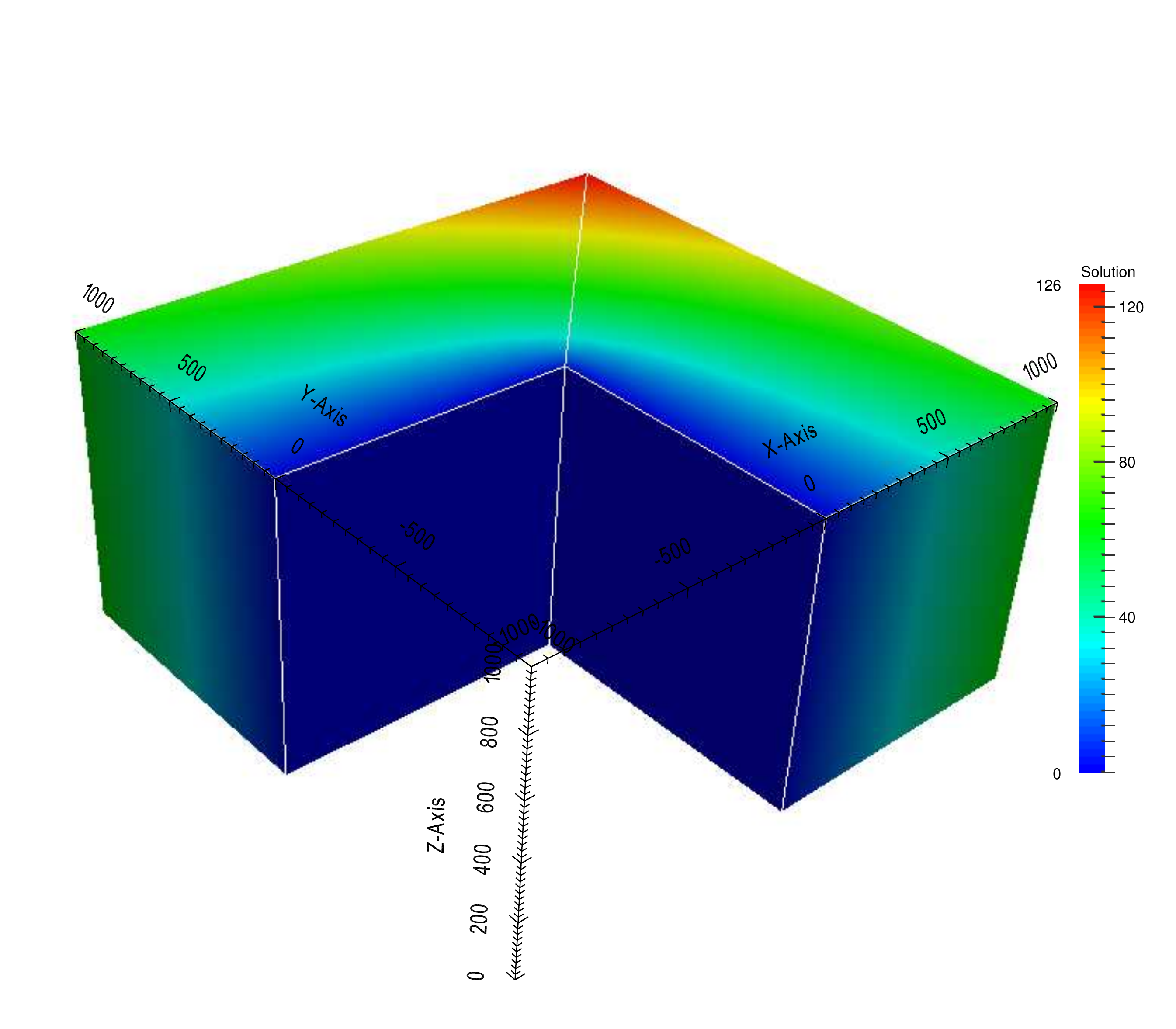}}
     \end{subfigmatrix}
	    \caption{3d L-shape test: (a) The domain $\Omega$ with the corners and the edge boundary singularities,
				      (b) The graded meshes of the two  subdomain,  
		                      (c) The contours of $u_h$.}
    \label{fig:sec:lshape3d}
    \end{figure}
    
    We consider an exact solution 
    given by $u = r^{\lambda}\sin(\frac{\theta \pi}{\omega}),$ where $\lambda= \pi / \omega$ and $\omega= 3 \pi /  2$. 
    We set $\Gamma_D=\partial \Omega$. The data ${f}$ and ${u}_D$ of (\ref{0}) are given
    by the   exact solution. 
    The computational domain $\Omega$ consists of two subdomains
    We have solved the problem using B-spline  spaces of order $k=1$ and $k=2$ using quasi-uniform and graded meshes
    in both subdomains.
    The grading parameter is defined by the relation $ \delta(l,p,d) /\mu =k$. 
    In Fig.~\ref{fig:sec:lshape3d}(b), we can see the 
    graded meshes for $\mu=0.6$. The contours of the  corresponding  approximate  solution $u_h$ computed
    for $k=1$ are presented in Fig.~\ref{fig:sec:lshape3d}(c).
    Table \ref{tab:sec:lshape3d} displays the convergence rates of the error. 
We observe the same behavior of the rates as in the previous examples. 
The rates of the uniform meshes are determined by the regularity of the solution 
    ($u\in W^{1+\lambda,p=2}(\Omega)$) for both B-spline  spaces. The convergence rates corresponding to graded meshes approach
    the optimal value. We remark here that the same  type of  graded meshes have also
    been used in finite element methods for approximating solutions of elliptic problems in three-dimensional domains
    with edges, see \cite{LMMT:ApelSandingWhiteman:1996,LMMT:ApelMilde1996,LMMT:Apel_NicaiseMatModelApplScin_1998}.
    \begin{table}[htb!]
       \centering %
         \begin{tabular}{|c|l|l|l|l|} 
           \hline 
	         &\multicolumn{2}{|l|}{without grading }
	         &\multicolumn{2}{|l|}{with grading } \\ [0.5ex] \hline 
           $h /2^s$    &$k=1$& $k=2$&$\begin{matrix}
		        	              k=1,\\ \mu=0.6
				            \end{matrix}$  
				    &$\begin{matrix}
					  k=2, \\ \mu=0.3
				      \end{matrix}$ \quad
				    \\ \hline 
	    \multicolumn{5}{|c|} {Convergence rates } \\ [0.5ex]  \hline
    $s=0$  &-        & -       & -        &-         \\ 
    $s=1$  &0.645078 &0.477178 &0.629909  &0.387338 {\ }   \\ 
    $s=2$  &0.650805 &0.639951 &0.869128  &1.11198  {\ }   \\ 
    $s=3$  &0.642971 &0.670841 &0.883655  &1.80531  {\ }   \\ 
    $s=4$  &0.644107 &0.669949 &0.902467  &1.96533  {\ }   \\ 
    $s=5$  &0.648100 &0.668371 &0.920065  &2.00296  {\ }   \\ \hline
    \end{tabular}
    \caption{3d L-shape : The convergence rates of the error with respect to the dG norm 
		  on  uniform and graded  meshes. }
    \label{tab:sec:lshape3d}
    \end{table}
    
    \subsubsection{Three-dimensional heart shaped domain.}
In this example, we consider an exact solution given by 
    $u = r^{\lambda}\sin(\theta \pi/\omega),$ 
    where $\lambda= \pi /\omega $ and $\omega= 3\pi/2$. 
    We again set $\Gamma_D=\partial \Omega$, and the data ${f}$ and ${u}_D$ of (\ref{0}) are 
    specified by the  given exact solution. 
    The computational domain $\Omega$ consists of two subdomains.
    The problem is solved with B-spline spaces of order $k=1$ 
    and $k=2$ using quasi-uniform and graded meshes in both subdomains.
    The grading parameter is defined by the relation $\lambda / \mu = k$.
    In Fig. \ref{fig:sec:heart3d}(b), we can see the graded meshes for $\mu=0.6.$ 
    The contours of the  corresponding approximate  solution $u_h$ computed with
    degree $k=1$ is presented in Fig.~\ref{fig:sec:heart3d}(c).

    The convergence rates of the error corresponding to the quasi-uniform meshes 
    are shown in left columns (without grading) of Table~\ref{tab:sec:heart3d},
    and the rates corresponding to the graded meshes are shown in the right 
    columns of Table~\ref{tab:sec:heart3d}.    
    \begin{figure}[thb!]
    \begin{subfigmatrix}{3}
     \subfigure[]{ \includegraphics[width=0.3\textwidth, height=0.22\textheight]{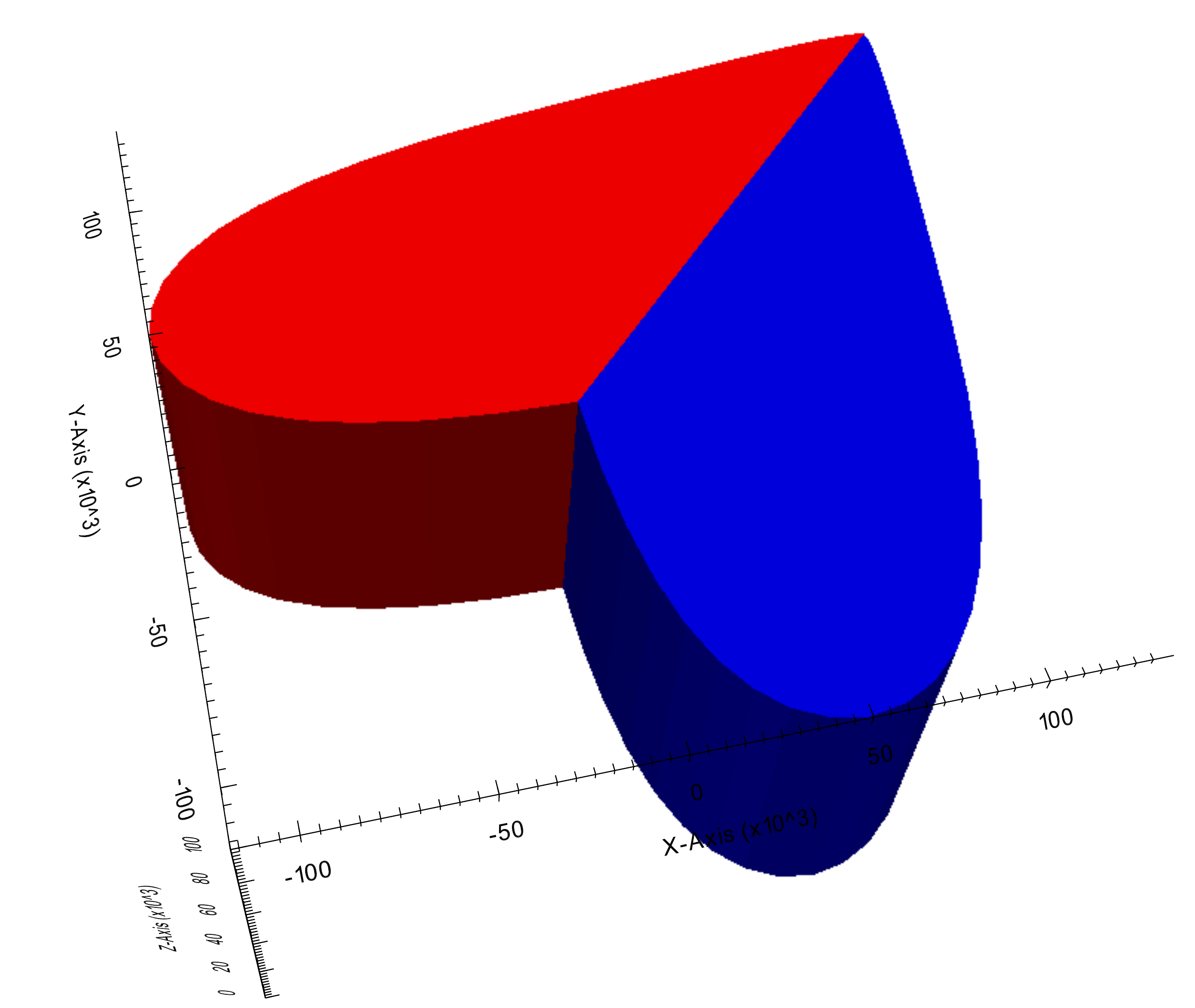}}
     \subfigure[]{ \includegraphics[width=0.3\textwidth, height=0.22\textheight]{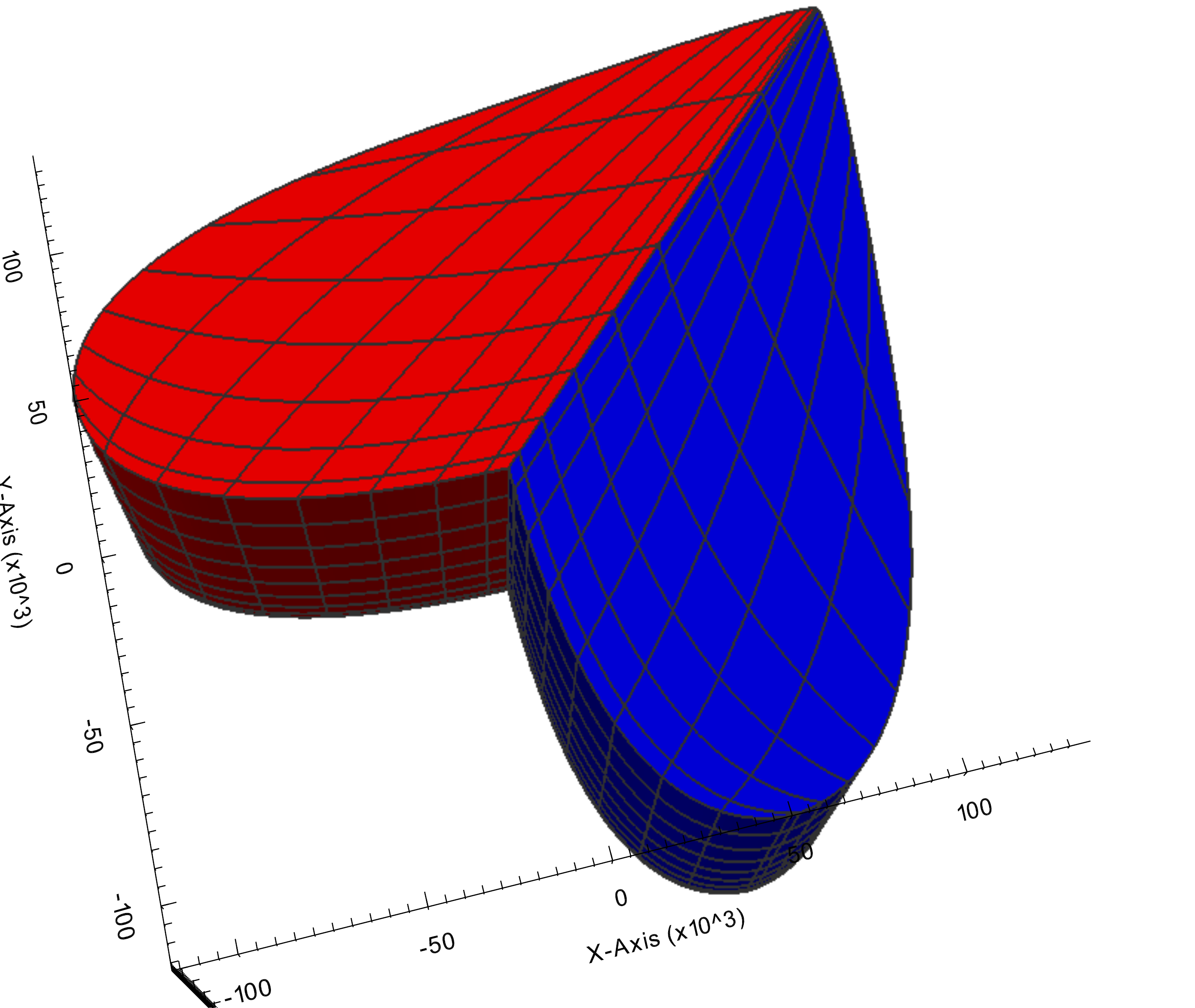}}
     \subfigure[]{ \includegraphics[width=0.3\textwidth, height=0.22\textheight]{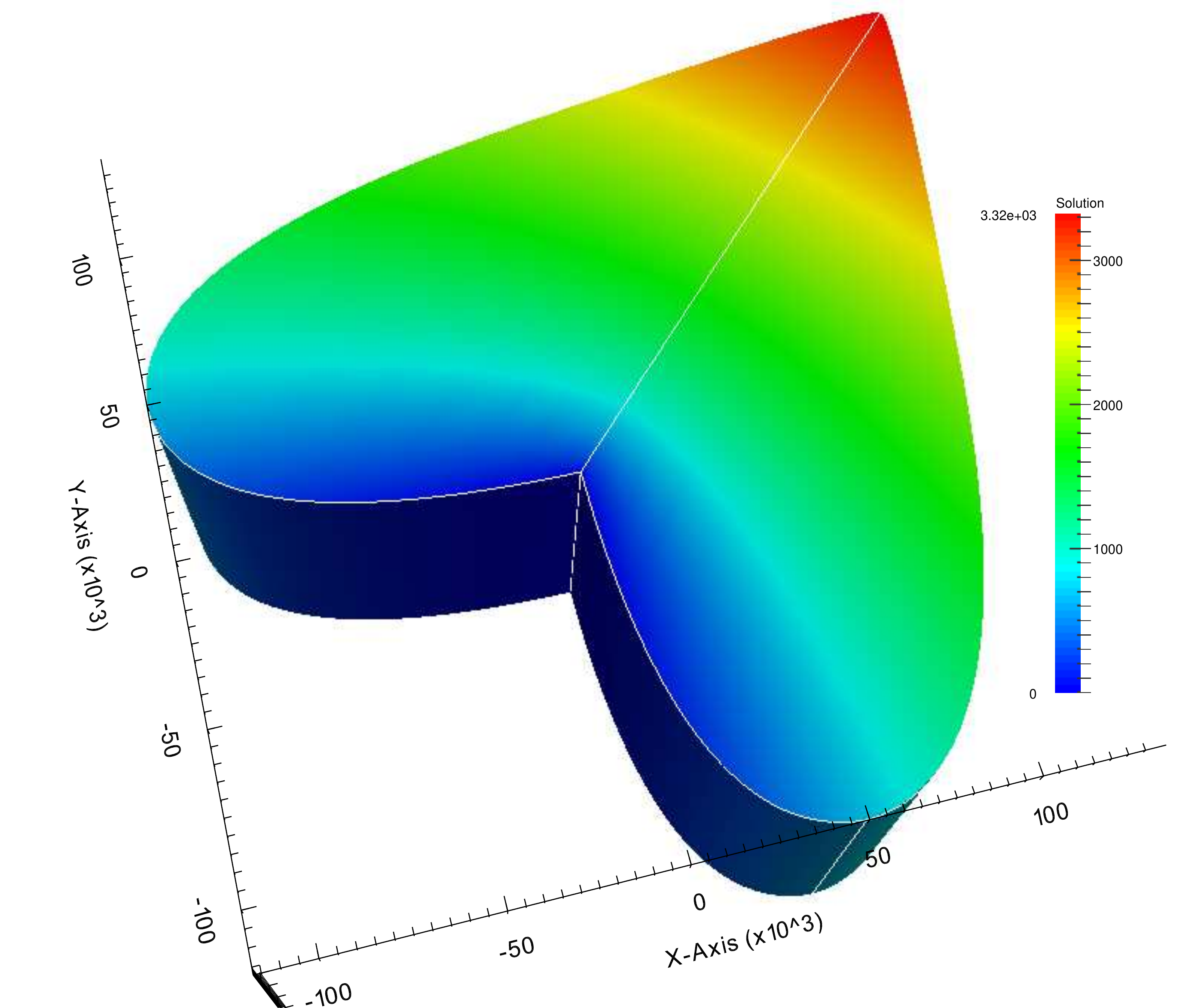}}
    \end{subfigmatrix}
	    \caption{3D heart test: (a) The domain $\Omega$ with the corners and edge boundary singularities,
				    (b) The graded meshes of the two  subdomain,  
				    (c) The contours of $u_h$.}
    \label{fig:sec:heart3d}
    \end{figure}    
    \begin{table}[htb!]
    \centering %
    \begin{tabular}{|c|l|l|l|l|} 
    \hline 
	      &\multicolumn{2}{|l|}{without grading }
	      &\multicolumn{2}{|l|}{with grading } \\ [0.5ex] \hline 
    $h /2^s$    &$k=1$& $k=2$&$\begin{matrix}
					  k=1,\\ \mu=0.6
				      \end{matrix}$  
				    &$\begin{matrix}
					  k=2, \\ \mu=0.3
				      \end{matrix}$ \quad
				    \\ \hline 
	    \multicolumn{5}{|c|} {Convergence rates } \\ [0.5ex]  \hline
    $s=0$  &-        & -        & -        &-         \\ 
    $s=1$  &0.650805 &0.675611  &0.686633  &0.964287 {\ }   \\ 
    $s=2$  &0.642971 &0.685756  &0.846524  &1.55143  {\ }   \\ 
    $s=3$  &0.644107 &0.674337  &0.902119  &1.91781  {\ }   \\ 
    $s=4$  &0.6481   &0.669817  &0.925762  &2.10561  {\ }   \\ 
    $s=5$  &0.65251  &0.667968  &0.94134   &2.09457  {\ }   \\ \hline
    \end{tabular}
    \caption{3d Heart : The convergence rates of the error  with respect to the dG norm on  uniform and graded  meshes.}
    \label{tab:sec:heart3d}
    \end{table}
    
%
%
    

\section{Conclusion}
We have presented  mesh grading  techniques    for dG IgA discretizions of 
elliptic boundary value problems in the presence of 
so-called
singular points. 
Based on the a priori or a posteriori knowledge 
of the behaviour of the exact solution around the singular points, 
we pre-defined the grading  of the mesh 
without increasing the knots but performing a relocation. 
The grading
refinement  has a subdomain (patch) character  in order to fit well into the IgA framework.
 Optimal error estimates of the 
multipatch dG IgA method
have been shown when it is 
used
on the  graded meshes proposed. 
The theoretical results
have been confirmed 
by 
a number of 
two- and three-dimensional test 
problems with known exact solutions.

\paragraph*{\bf Acknowledgment}
%
%
%
This research was supported by the National Research Network NFN S117-03
``Geometry \textbf{+} Simulation'' of the Austrian Science Fund (FWF).


\bibliographystyle{abbrv} 
\bibliography{MeshGrading}

\begin{thebibliography}{10}

\bibitem{LMMT:Adams_Sobolevbook}
R.~A. Adams and J.~J.~F. Fournier.
\newblock {\em Sobolev {S}paces}, volume 140 of {\em Pure and {A}pplied
  {M}athematics}.
\newblock {ACADEMIC PRESS}-imprint {E}lsevier {S}cience, second edition, 2003.

\bibitem{LMMT:Apel_1999_Int_NonSmooth}
T.~Apel.
\newblock Interpolation of non-smooth functions on anisotropic finite element
  meshes.
\newblock {\em M2AN}, 33(6):1149--1185, 1999.

\bibitem{LMMT:ApelHeinrich:1994a}
T.~Apel and B.~Heinrich.
\newblock Mesh refinement and windowing near edges for some elliptic problem.
\newblock {\em SIAM J. Numer. Anal.}, 31(3):695--708, 1994.

\bibitem{LMMT:Apel_NicaiseMatModelApplScin_1998}
T.~Apel and B.~Heinrich.
\newblock The finite element method with anisotropic mesh grading for elliptic
  problems in domains with corner and edges.
\newblock {\em SIAM J. Numer. Anal.}, 31(3):695--708, 1998.

\bibitem{LMMT:ApelMilde1996}
T.~Apel and F.~Milde.
\newblock Comparison of several mesh refinement strategies near edges.
\newblock {\em Comm. Num. Meth. Eng.}, 12:373--381, 1996.

\bibitem{LMMT:ApelSandingWhiteman:1996}
T.~Apel, A.-M. S\"andig, and J.~R. Whiteman.
\newblock Graded mesh refinement and error estimates for finite element
  solutions of elliptic boundary value problems in non-smooth domains.
\newblock {\em Math. Methods Appl. Sci.}, 19(30):63--85, 1996.

\bibitem{LMMT:BazilevsBeiraoCottrellHughesSangalli:2006a}
Y.~Bazilevs, L.~Beir\~ao~da Veiga, J.~Cottrell, T.~Hughes, and G.~Sangalli.
\newblock Isogeometric analysis: {A}pproximation, stability and error estimates
  for $h$-refined meshes.
\newblock {\em M3AS}, 16(07):1031--1090, 2006.

\bibitem{LMMT:SangalliaNurbs2012}
L.~Beir\~ao~da Veiga, D.~Cho, and G.~Sangalli.
\newblock Anisotropic {NURBS} approximation in isogeometric analysis.
\newblock {\em Comp. Methods in Appl. Mech and Engrg}, 209–212(0):1 -- 11,
  2012.

\bibitem{LMMT:Hughes_IGABook}
J.~A. Cottrell, T.~J.~R. Hughes, and Y.~Bazilevs.
\newblock {\em Isogeometric {A}nalysis, Toward {I}ntegration of {CAD} and
  {FEA}}.
\newblock John Wiley and Sons, 2009.

\bibitem{LMMT:DiPietroErn:2012a}
D.~A. Di~Pietro and A.~Ern.
\newblock {\em Mathematical Aspects of Discontinuous Galerkin Methods},
  volume~69 of {\em Math\'{e}matiques et Applications}.
\newblock Springer-Verlag, Heidelberg, Dordrecht, London, New York, 2012.

\bibitem{LMMT:Dryja:2003a}
M.~Dryja.
\newblock On discontinuous {G}alerkin methods for elliptic problems with
  discontinuous coeffcients.
\newblock {\em Comput. Methods Appl. Math.}, 3:76--85, 2003.

\bibitem{LMMT:RICHARD_FALKAND_OSBORN_MixedEll_1994}
S.~R. Falkand and J.~E. Osborn.
\newblock Remarks on mixed finite element methods for problems with rough
  coefficients.
\newblock {\em Math. Comp.}, 62(205):1--19, 1994.

\bibitem{LMMT:FeistauerSaendig:2012a}
M.~Feistauer and A.-M. S\"andig.
\newblock Graded mesh refinement and error estimates of higher order for dgfe
  solutions of elliptic boundary value problems in polygons.
\newblock {\em Numer. Methods Partial Diff. Equations}, 28(4):1124--1151, 2012.

\bibitem{LMMT:Grisvard:1985a}
P.~Grisvard.
\newblock {\em Elliptic problems in nonsmooth domains}.
\newblock Monographs and studies in mathematics. Pitman Advanced Pub. Program,
  1985.

\bibitem{LMMT:Grisvard:1992a}
P.~Grisvard.
\newblock {\em Singularities in Boundary Value Problems}.
\newblock Recherches en math{\'e}matiques appliqu{\'e}es. Masson, 1992.

\bibitem{LMMT:HughesCottrellBazilevs:2005a}
T.~Hughes, J.~Cottrell, and Y.~Bazilevs.
\newblock Isogeometric analysis: {CAD}, finite elements, {NURBS}, exact
  geometry and mesh refinement.
\newblock {\em Comput. Methods Appl. Mech. Engrg.}, 194:4135--4195, 2005.

\bibitem{LMMT:JeongOhKangKim:2013a}
J.~W. Jeong, H.~S. Oh, S.~K., and H.~Kim.
\newblock Mapping techniques for isogeometric analysis of elliptic boundary
  value problems containing singularities.
\newblock {\em Comp. Methods in Appl. Mech and Engrg}, 254(0):334 -- 352, 2013.

\bibitem{LMMT:Kellog_DiscDifCoef_1975}
R.~B. Kellogg.
\newblock On the {P}oisson equation with intersecting interfaces.
\newblock {\em Appl. Anal.}, 4:101--129, 1975.

\bibitem{LMMT:Kondratev:1967a}
V.~A. Kondrat'ev.
\newblock Boundary value problems for elliptic equations in domains with
  conical or angular points.
\newblock {\em Transl. Moscow Math. Soc.}, 16:227--313, 1967.

\bibitem{LMMT:KozlovMazyaRossmann2001}
V.~Kozlov, V.~G. Maz'ya, and J.~Rossmann.
\newblock {\em Spectral {P}roblems {A}ssociated with {C}orner {S}ingularities
  of {S}olutions to {E}lliptic {E}quations}, volume~85 of {\em Mathematical
  Surveys and Monographs}.
\newblock Americanb Mathematical Society, Rhode Issland, USA, 2001.

\bibitem{LMMT:LDG_pLaplace_IT_2014}
D.~Kr{\"o}ner, M.~R\r{u}\v{z}i\v{c}ka, and I.~Toulopoulos.
\newblock Numerical solutions of systems with $(p,\delta)$-structure using
  local discontinuous {G}alerkin finite element methods.
\newblock {\em Int. J. Numer. Methods Fluids}, 2014.

\bibitem{LMMT:LangerMantzaflarisMooreToulopoulos:2014a}
U.~Langer, A.~Mantzaflaris, S.~Moore, and I.~Toulopoulos.
\newblock Multipatch discontinuous {G}alerkin {I}sogeometric {A}nalysis.
\newblock RICAM Reports 2014-18, Johann Radon Institute for Computational and
  Applied Mathematics, Austrian Academy of Sciences, Linz, 2014.
\newblock http://arxiv.org/abs/1411.2478v1.

\bibitem{LMMT:LangerMoore:2014a}
U.~Langer and S.~Moore.
\newblock Discontinuous {G}alerkin isogeometric analysis of elliptic {PDE}s on
  surfaces.
\newblock NFN Technical Report~12, Johannes Kepler University Linz, NFN
  Geometry and Simulation, Linz, 2014.
\newblock http://arxiv.org/abs/1402.1185 and accepted for publication in the
  DD22 proceedings.

\bibitem{LMMT:LangerToulopoulos:2014a}
U.~Langer and I.~Toulopoulos.
\newblock Analysis of multipatch discontinuous {G}alerkin {IgA} approximations
  to elliptic boundary value problems.
\newblock RICAM Reports 2014-08, Johann Radon Institute for Computational and
  Applied Mathematics, Austrian Academy of Sciences, Linz, 2014.
\newblock http://arxiv.org/abs/1408.0182.

\bibitem{LMMT:OganesjanRuchovetz:1979a}
L.~Oganesjan and L.~Ruchovetz.
\newblock {\em Variational {D}ifference {M}ethods for the {S}olution of
  {E}lliptic {E}quations}.
\newblock Isdatelstvo Akademi Nank Armjanskoj SSR, Erevan, 1979.
\newblock (in Russian).

\bibitem{LMMT:OhKimJeong:2013a}
H.~S. Oh, H.~Kim, and J.~W. Jeong.
\newblock Enriched isogeometric analysis of elliptic boundary value problems in
  domains with cracks and/or corners.
\newblock {\em Int. J. Numer. Meth. Engn}, 97(3):149--180, 2014.

\bibitem{LMMT:Riviere:2008a}
B.~Rivi\`ere.
\newblock {\em Discontinuous Galerkin Methods for Solving Elliptic and
  Parabolic Equations: Theory and Implementation}.
\newblock SIAM, Philadelphia, 2008.

\bibitem{LMMT:Schumaker_Bspline_book}
L.~L. Schumaker.
\newblock {\em Spline {F}unctions: {B}asic {T}heory}.
\newblock Cambridge, University Press, 3rd edition, 2007.

\bibitem{LMMT:StrangFix:1973a}
G.~Strang and G.~Fix.
\newblock {\em An Analysis of the Finite Element Method}.
\newblock Prentice-Hall. Englewood Cliffs, N.J., 1973.

\end{thebibliography}

\end{document}